\input ./style/arxiv-vmsta.cfg
\documentclass[numbers,compress,v1.0.1]{vmsta}
\usepackage{vtexbibtags}
\usepackage{vmkcol}
\usepackage{mathbh}
\usepackage{authorquery}

\volume{5}
\issue{3}
\pubyear{2018}
\firstpage{353}
\lastpage{383}
\aid{VMSTA116}
\doi{10.15559/18-VMSTA116}
\articletype{research-article}

\startlocaldefs
\theoremstyle{plain}
\newtheorem{theorem}{Theorem}[section]
\newtheorem{lemma}[theorem]{Lemma}
\newtheorem{corollary}[theorem]{Corollary}
\newtheorem{proposition}[theorem]{Proposition}

\theoremstyle{definition}
\newtheorem{remark}[theorem]{Remark}
\newtheorem{example}[theorem]{Example}
\newtheorem{definition}[theorem]{Definition}
\newtheorem*{notation}{Notation}

\newcommand{\compp}{\mathbb{C}}
\newcommand{\Ee}{\mathbb{E}}
\newcommand{\nat}{\mathbb{N}}
\newcommand{\real}{\mathbb{R}}
\newcommand{\Pp}{\mathbb{P}}
\newcommand{\I}{\mathbh{1}}
\newcommand{\G}{\mathbb{G}}
\newcommand{\rdohne}{{\mathbb{R}^d\setminus\{0\}}}
\newcommand{\rmohne}{{\mathbb{R}^m\setminus\{0\}}}
\newcommand{\rnohne}{{\mathbb{R}^n\setminus\{0\}}}
\newcommand{\rd}{{\mathbb{R}^d}}

\newcommand{\VN}{\mbox{}^{\scriptscriptstyle N}\kern-1.5pt{V}}
\newcommand{\FN}{\mbox{}^{\scriptscriptstyle N}\kern-1.5pt{F}}
\newcommand{\RN}{\mbox{}^{\scriptscriptstyle N}\kern-2pt{R}}

\newcommand{\dd}{\mathrm{d}}
\newcommand{\ddr}{\mathrm{d}r}
\newcommand{\dds}{\mathrm{d}s}
\newcommand{\ddt}{\mathrm{d}t}
\newcommand{\ddu}{\mathrm{d}u}
\newcommand{\ddx}{\mathrm{d}x}
\newcommand{\ddy}{\mathrm{d}y}

\newcommand{\scalp}[2]{#1\cdot#2} 
\newcommand{\et}{\quad\text{and}\quad}

\newcommand{\law}{\mathcal{L}}

\newcommand{\Cskript}{\mathcal{C}}

\newcommand\doverline[1]{\overline{\overline{#1}}}
\def\arxivhref#1{\href{https://arxiv.org/abs/#1}{#1}}

\newcommand{\ii}{\mathrm{i}}
\newcommand{\ee}{\mathrm{e}}

\renewcommand{\Re}{\mathop{\mathrm{Re}}}

\newcommand{\rrVert}{\Vert}
\newcommand{\llVert}{\Vert}
\newcommand{\rrvert}{\vert}
\newcommand{\llvert}{\vert}
\allowdisplaybreaks
\endlocaldefs

\DeclareMathOperator{\Cov}{Cov}
\DeclareMathOperator{\Cor}{Cor}

\DeclareMathOperator{\trace}{trace}

\allowdisplaybreaks

\begin{aqf}
\querytext{Q1}{I believe, it would be more precise to say here: ...cndf $\varTheta$, ...}
\querytext{Q2}{Do 'latter' means 'computability'? If so, saying it explicitly would make the sentence more clear.}
\querytext{Q3}{Looking to the following formulas, I believe "with easy calculations" would sound better than "without further calculations".}
\querytext{Q4}{Please check the edited sentence.}
\end{aqf}
\begin{document}

\begin{frontmatter}
\pretitle{Research Article}

\title{Detecting independence of random vectors: generalized distance covariance and Gaussian covariance}

\author{\inits{B.}\fnms{Bj\"orn}~\snm{B\"ottcher}\thanksref{cor1}\ead[label=e1]{bjoern.boettcher@tu-dresden.de}}
\author{\inits{M.}\fnms{Martin}~\snm{Keller-Ressel}\ead[label=e2]{martin.keller-ressel@tu-dresden.de}}
\author{\inits{R.L.}\fnms{Ren\'e~L.}~\snm{Schilling}\ead[label=e3]{rene.schilling@tu-dresden.de}}
\thankstext[type=corresp,id=cor1]{Corresponding author.}
\address{\institution{TU Dresden}, Fakult\"at Mathematik, Institut f\"{u}r Mathematische Stochastik, 01062 Dresden, \cny{Germany}}



\markboth{B. B\"ottcher et al.}{Generalized distance covariance and Gaussian covariance}

\begin{abstract}
Distance covariance is a quantity to measure the dependence of two random
vectors. We show that the original concept introduced and developed by
Sz\'{e}kely, Rizzo and Bakirov can be embedded into a more general
framework based on symmetric L\'evy measures and the corresponding
real-valued continuous negative definite functions. The L\'evy measures
replace the weight functions used in the original definition of distance
covariance. All essential properties of distance covariance are preserved
in this new framework.

From a practical point of view this allows less restrictive moment
conditions on the underlying random variables and one can use other
distance functions than Euclidean distance, e.g. Minkowski distance. Most
importantly, it serves as the basic building block for distance
multivariance, a quantity to measure and estimate dependence of multiple
random vectors, which is introduced in a follow-up paper [Distance Multivariance:
New dependence measures for random vectors (submitted).
Revised version of arXiv: \arxivhref{1711.07775v1}] 
to the present article.
\end{abstract}

\received{\sday{29} \smonth{6} \syear{2018}}
\revised{\sday{23} \smonth{8} \syear{2018}}
\accepted{\sday{30} \smonth{8} \syear{2018}}
\publishedonline{\sday{19} \smonth{9} \syear{2018}}

\end{frontmatter}

\section{Introduction}
The concept of distance covariance was introduced by Sz\'ekely, Rizzo and
Bakirov \cite{SzekRizzBaki2007} as a measure of dependence between two random
vectors of arbitrary dimensions. Their starting point is to consider a
weighted $L^2$-integral of the difference of the (joint) characteristic
functions $f_X, f_Y$ and $f_{(X,Y)}$ of the ($\real^m$- and $\real^n$-valued)
random variables $X$, $Y$ and $(X,Y)$,
\begin{equation}
\label{eq:Vrep_SR} \mathcal V^2(X,Y;w) = \iint_{\real^{m+n}}
|f_{(X,Y)}(s,t)-f_X(s)f_Y(t)|^2
\,w(s,t)\,\dds\,\ddt.
\end{equation}
The weight
$w$ is given by $w(s,t) := c_{\alpha,m}|s|^{-m-\alpha} c_{\alpha,n}
|t|^{-n-\alpha}$ for $\alpha\in (0,2)$.

We are going to embed this into a more general framework. In order to
illustrate the new features of our approach we need to recall some results on
distance covariance. Among several other interesting properties,
\cite{SzekRizzBaki2007} shows that distance covariance characterizes
independence, in the sense that $V^2(X,Y;w) = 0$ if, and only if, $X$ and $Y$
are independent. Moreover, they show that in the case $\alpha = 1$ the
distance covariance $\VN^2(X,Y;w)$ of the empirical distributions of two
samples $(x_1, x_2, \dots, x_N)$ and $(y_1, y_2, \dots, y_N)$ takes a
surprisingly simple form. It can be represented as
\begin{equation}
\label{eq:Vrep_SR2} \VN^2(X,Y;w) = \frac{1}{N^2} \sum
_{k,l = 1}^N A_{kl} \xch{B_{kl},}{B_{kl}}
\end{equation}
where $A$ and $B$ are double centrings (cf. Lemma~\ref{est-05}) of the
Euclidean distance matrices of the samples, i.e. of $(|x_k -
x_l|)_{k,l=1,\ldots,N}$ and $(|y_k - y_l|)_{k,l=1,\ldots,N}$.
 If \mbox{$\alpha \neq 1$} then Euclidean distance has to be replaced by its power with exponent $\alpha$. The connection between the weight function $w$ in \eqref{eq:Vrep_SR} and the (centred) Euclidean distance matrices in \eqref{eq:Vrep_SR2} is given by the L\'evy--Khintchine representation of negative definite functions, i.e.
\begin{gather*}
|x|^\alpha = c_p \int_{\rmohne} (1-\cos
\scalp{s} {x})\,\frac{\dds}{|s|^{m+\alpha}},\quad x\in\real^m,
\end{gather*}
where $c_p$ is a suitable constant, cf. Section~\ref{ar-cndf},
Table~\ref{fig-cndf}. Finally, the representation \eqref{eq:Vrep_SR2} of
$\VN^2(X,Y;w)$ and its asymptotic properties as $N \to \infty$ are used by
Sz\'ekely, Rizzo and Bakirov to develop a statistical test for independence
in \cite{SzekRizzBaki2007}.

Yet another interesting representation of distance covariance is given in the
follow-up paper \cite{SzekRizz2009}:  Let $(X_{\text{cop}},Y_{\text{cop}})$
be an independent copy of $(X,Y)$ and let $W$ and $W'$ be Brownian random
fields on $\real^m$ and $\real^n$, independent from each other and from
$X,Y,X_{\text{cop}},Y_{\text{cop}}$. The paper \cite{SzekRizz2009} defines the
\emph{Brownian covariance}
\begin{equation}
\label{eq:Vrep_SR3} \mathcal W^2(X,Y) = \Ee \bigl[X^W
X^W_{\text{cop}} Y^{W'} Y^{W'}_{\text{cop}}
\bigr],
\end{equation}
where $X^W := W(X) - \Ee [W(X) \mid W ]$ for any random variable $X$
and random field $W$ with matching dimensions. Surprisingly, as shown in
\cite{SzekRizz2009}, Brownian covariance coincides with distance covariance,
i.e. $\mathcal W^2(X,Y) = \mathcal V^2(X,Y;w)$ when $\alpha = 1$ is chosen
for the kernel $w$.

The paper \cite{SzekRizz2009} was accompanied by a series of discussion
papers
\cite{Newt2009,BickXu2009,Koso2009,Cope2009,Feue2009,GretFukuSrip2009,Remi2009,Geno2009,SzekRizz2009a}
where various extensions, applications and open questions were suggested. Let
us highlight the three problems which we are going to address:
\begin{enumerate}
\item
    Can the weight function $w$ in \eqref{eq:Vrep_SR} be replaced by other
    weight functions?\break (Cf.~\cite{Feue2009,GretFukuSrip2009})
\item
    Can the Euclidean distance (or its $\alpha$-power) in
    \eqref{eq:Vrep_SR2} be replaced by other distances?
    (Cf.~\cite{Koso2009,GretFukuSrip2009,Lyon2013})
\item
    Can the Brownian random fields $W, W'$ in \eqref{eq:Vrep_SR3} be
    replaced by other random fields? (Cf.~\cite{Koso2009,Remi2009})
\end{enumerate}
While insights and partial results on these questions can be found in all of
the mentioned discussion papers, a definitive and unifying answer was missing
for a long time. In the present paper we propose a generalization of distance
covariance which resolves these closely related questions. In a follow-up
paper \cite{part2} we extend our results to the detection of independence of
$d$ random variables $(X^1, X^2, \dots, X^d)$, answering a question of
\cite{Feue2009,BakiSzek2011}.

More precisely, we introduce in Definition~\ref{gdc-05} the
\emph{generalized} distance covariance
\begin{equation*}
V^2(X,Y) = \int_{\real^n}\int_{\real^m}
|f_{(X,Y)}(s,t)-f_X(s)f_Y(t)|^2 \,
\mu(\dds)\,\xch{\nu(\ddt),}{\nu(\ddt)}
\end{equation*}
where $\mu$ and $\nu$ are symmetric L\'evy measures, as a natural extension
of distance covariance of Sz\'ekely et al. \cite{SzekRizzBaki2007}. The
L\'evy measures $\mu$ and $\nu$ are linked to negative definite functions
$\varPhi$ and $\varPsi$ by the well-known L\'evy--Khintchine representation,
cf. Section~\ref{ar} where examples and important properties of negative
definite functions are discussed. In Section~\ref{sec:gdc} we show that
several different representations (related to \cite{Lyon2013}) of $V^2(X,Y)$
in terms of the functions $\varPhi$ and $\varPsi$ can be given. In
Section~\ref{sec:estimation} we turn to the finite-sample properties of
generalized distance covariance and show that the representation
\eqref{eq:Vrep_SR2} of $\VN^2(X,Y)$ remains valid, with the Euclidean
distance matrices replaced by the matrices
\begin{gather*}
\bigl(\varPhi(x_k - x_l) \bigr)_{k,l = 1,\ldots,N}
\quad \text{and} \quad \bigl(\varPsi(y_k - y_l) \bigr)_{k,l = 1,\ldots,N}.
\end{gather*}
We also show asymptotic properties of $\VN^2(X,Y)$ as $N$ tends to infinity,
paralleling those of \cite{SzekRizz2009,SzekRizz2012} for Euclidean distance
covariance. After some remarks on uniqueness and normalization, we show in
Section~\ref{gauss} that the representation \eqref{eq:Vrep_SR3} remains also
valid, when the Brownian random fields $W$ and $W'$ are replaced by centered
Gaussian random fields $G_\varPhi$ and $G_\varPsi$ with covariance kernel
\begin{gather*}
\Ee \bigl[G_\varPhi(x)G_\varPhi\bigl(x'\bigr)
\bigr] = \varPhi(x) + \varPhi\bigl(x'\bigr) - \varPhi\bigl(x -
x'\bigr)
\end{gather*}
and analogously for $G_\varPsi$.

To use generalized distance covariance (and distance multivariance) in
applications all necessary functions and tests are provided in the R package
\texttt{multivariance} \cite{Boett2017R-1.0.5}. Extensive examples and
simulations can be found in \cite{Boet2017}, therefore we concentrate in the
current paper on the theoretical foundations.

\begin{notation}
Most of our notation is standard or self-explanatory.
Throughout we use positive (and negative) in the non-strict sense, i.e. $x\geq 0$ (resp. $x\leq 0$) and we write $a\vee b = \max\{a,b\}$ and
$a\wedge b = \min\{a,b\}$ for the maximum and minimum. For a vector $x\in
\real^d$ the Euclidean norm is denoted by $|x|$.
\end{notation}

\section{Fundamental results}\label{ar}

In this section we collect some tools and concepts which will be needed in
the sequel.

\subsection{Negative definite functions}\label{ar-cndf}

A function $\varTheta : \real^d \to \compp$ is called \emph{negative
definite} (in the sense of Schoenberg) if the matrix
$ (\varTheta(x_i)+\overline{\varTheta(x_j)} -
\varTheta(x_i-x_j) )_{i,j}\in\compp^{m\times m}$ is positive semidefinite hermitian for every $m\in\nat$
and $x_1,\dots,x_m\in\real^d$. It is not
hard to see, cf. Berg \& Forst \cite{BergFors75} or Jacob \cite{Jaco2001},
that this is equivalent to saying that $\varTheta(0)\geq 0$, $\varTheta(-x) =
\overline{\varTheta(x)}$ and the matrix
$ (-\varTheta(x_i-x_j) )_{i,j}\in\compp^{m\times m}$ is
conditionally positive definite, i.e.
\begin{gather*}
\sum_{i,j=1}^m \bigl[-
\varTheta(x_i-x_j)\bigr]\lambda_i\bar
\lambda_j\geq 0 \quad\forall \lambda_1,\dots,
\lambda_m\in\compp\text{\ such that\ } \sum
_{k=1}^m \lambda_k = 0.
\end{gather*}
Because of this equivalence, the function $-\varTheta$ is also called
\emph{conditionally positive definite} (and some authors call $\varTheta$
\emph{conditionally negative definite}).

Negative definite functions appear naturally in several contexts, for
instance in probability theory as characteristic exponents (i.e. logarithms
of characteristic functions) of infinitely divisible laws or L\'evy
processes, cf. Sato~\cite{sato} or \cite{BoetSchiWang2013}, in harmonic
analysis in connection with non-local operators,  cf. Berg \&
Forst~\cite{BergFors75} or Jacob~\cite{Jaco2001} and in geometry when it
comes to characterize certain metrics in Euclidean spaces, cf. Benyamini \&
Lindenstrauss~\cite{BenLin}.

The following theorem, compiled from Berg \& Forst \cite[Sec.~7, pp.~39--48,\break
Thm.~10.8, p.~75]{BergFors75} and Jacob \cite[Sec.~3.6--7,
pp.~120--155]{Jaco2001}, summarizes some basic equivalences and connections.
\begin{theorem}\label{ar-03}
    For a function $\varTheta:\real^d\to\compp$ with $\varTheta(0)=0$ the following assertions are equivalent
\begin{enumerate}
    \item\label{ar-03-a} $\varTheta$ is negative definite.
    \item\label{ar-03-b} $-\varTheta$ is conditionally positive definite.
    \item\label{ar-03-c} $\ee^{-t\varTheta}$ is positive definite for every
        $t>0$.
    \item\label{ar-03-d} $t^{-1}(1- \ee^{-t\varTheta})$ is negative
        definite for every $t>0$.
\end{enumerate}
If $\varTheta$ is continuous, the assertions \ref{ar-03-a}--\ref{ar-03-d} are also equivalent to
\begin{enumerate}\setcounter{enumi}{4}
    \item\label{ar-03-e} $\varTheta$ has the following integral
        representation
\begin{align}
\varTheta(x) &= \ii\scalp lx + \frac{1}2\scalp{x} {Qx} \nonumber\\
&\quad +
\int_{\rdohne} \bigl(1-\ee^{\ii\scalp{x}{r}} + \ii\scalp{x} {r}
\I_{(0,1)}(|r|) \bigr)\,\xch{\rho(\ddr),}{\rho(\ddr)}\label{eq:lkf}
\end{align}
        where $l\in\real^d$, $Q\in\real^{d\times d}$ is symmetric and
        positive semidefinite and $\rho$ is a measure on $\rdohne$ such
        that $\int_{\rdohne}  (1\wedge
        |s|^2 )\,\rho(\dds)<\infty$.
    \end{enumerate}
\end{theorem}
We will frequently use the abbreviation \emph{cndf} instead of
\emph{continuous negative definite function}. The representation
\eqref{eq:lkf} is the \emph{L\'evy--Khintchine formula} and any measure
$\rho$ satisfying
\begin{equation}
\label{ar-e08} \rho\text{\ \ is a measure on $\rdohne$ such that\ \ } \int
_{\rdohne} \bigl(1\wedge |r|^2 \bigr)\,\rho(\ddr)<
\infty
\end{equation}
is commonly called \emph{L\'evy measure}. To keep notation simple, we will
write $\int \cdots\rho(\ddr)$ or $\int_{\real^d}\cdots\rho(\ddr)$ instead of
the more precise $\int_{\real^d\setminus\{0\}}\cdots\rho(\ddr)$.

\begin{table}[t!]
\begin{center}
\renewcommand{\arraystretch}{1.8}
\caption{Some real-valued continuous negative definite functions (cndfs) on $\real^d$ and the corresponding L\'evy measures and infinitely divisible distributions (IDD)}
\label{fig-cndf}
\begin{tabular}{p{.3\linewidth}|c|c}
	cndf $\Theta(x)$
&            L\'evy measure  $\rho(\ddr)$
&                      IDD
\\ \hline
	$\int \left(1-\cos \scalp xr\right)\, \rho(\ddr)$
&               $\rho$ finite measure
&             compound Poisson (CP)
\\
	$\frac{|x|^2}{\lambda^2+|x|^2}$, $\lambda>0$, $x\in\real$
&   $2\lambda^{-1} \ee^{-\lambda|r|}\,\ddr$
&   CP with exponential (1-d)
\\
	$\frac{|x|^2}{\lambda^2+|x|^2}$, $\lambda>0$, $x\in\real^d$
&   $\int_0^\infty \ee^{\lambda^2 u-\frac{|r|^2}{4u}}\,\frac{\lambda^2\,\ddu}{(4\pi u)^{d/2}}\,\ddr$
&    CP, cf.~\cite[Lem.~6.1]{SchiSchn09}
\\
	$1-\ee^{-\frac{1}{2}|x|^2}$
&   $(2\pi)^{-d/2} \ee^{-\frac{1}{2} |r|^2}\,\ddr$
& CP with normal
\\
	$|x|$
&   $\frac{\Gamma\left(\tfrac{1+d}{2}\right)}{\pi^{\frac{1+d}2}}\dfrac{\ddr}{|r|^{d+1}}$
&   Cauchy
\\
	$\frac{1}{2}|x|^2$
&   no L\'evy measure
&   normal
\\
	$ |x|^\alpha,\ \alpha\in(0,2)$
&   $\frac{\alpha 2^{\alpha-1} \Gamma\left(\tfrac{\alpha+d}{2}\right)}{\pi^{d/2} \Gamma\left(1-\tfrac{\alpha}{2}\right)}\dfrac{\ddr}{|r|^{d+\alpha}}$
&   $\alpha$-stable
\\ $\sqrt[p]{\sum_{k=1}^d |x_k|^p}, p\in [1,2]$
& see Lemma \ref{lem:p-minkowski}
& see Lemma \ref{lem:p-minkowski}
\\
	$\ln(1+\frac{x^2}{2})$, $x\in\real$
&   $\dfrac{1}{|r|}\ee^{-\sqrt{\frac{1}{2}} |r|}\,\ddr$
&   variance Gamma (1-d)
\\
	$\ln \cosh (x)$, $x\in\real$
&   $\dfrac{\ddr}{2r \sinh(\pi r/2)}$
&   Meixner (1-d)
\\
	$|x|^\alpha + |x|^\beta,\ \alpha,\beta\in(0,2)$
&
&               mixture of stable
\\
	\multicolumn{1}{L{60pt}|}{$(1+|x|^\alpha)^{\frac{1}{\beta}} - 1$, $\alpha \in (0,2)$, $\beta \geq \frac{\alpha}{2}$}
&
&              relativistic stable
\end{tabular}
\end{center}
\end{table}

The triplet $(l,Q,\rho)$ uniquely determines $\varTheta$; moreover
$\varTheta$ is real (hence, positive) if, and only if, $l=0$ and $\rho$ is
symmetric, i.e. $\rho(B) = \rho(-B)$ for any Borel set $B\subset\rdohne$. In
this case \eqref{eq:lkf} becomes
\begin{equation}
\label{eq:lkf-sy} \varTheta(x) = \frac{1}2\scalp{x} {Qx} + \int
_{\rd} (1-\cos\scalp{x} {r} )\,\rho(\ddr).
\end{equation}

Using the representation \eqref{eq:lkf} it is straightforward to see that we
have $\sup_x|\varTheta(x)|<\infty$ if $\rho$ is a finite measure, i.e.,
$\rho(\rdohne)<\infty$, and $Q = 0$. The converse is also true, see
\cite[pp.~1390--1391, Lem.~6.2]{SchiSchn09}.

Table~\ref{fig-cndf} contains some examples of continuous negative definite
functions along with the corresponding L\'evy measures and infinitely
divisible laws.

A measure $\rho$ on a topological space $X$ is said to have  \emph{full}
(\emph{topological}) \emph{support}, if $\mu(G)>0$ for any open set $G\subset
X$; for L\'evy measures we have $X = \real^d\setminus\{0\}$.
\begin{lemma}\label{lem:p-minkowski}
    Let $p\in [1,2]$. The Minkowski distance function
\begin{gather*}
\ell_p(x) := \bigl(|x_1|^p + \cdots +
|x_d|^p \bigr)^{1/p},\quad x=(x_1,
\dots,x_d)\in\real^d
\end{gather*}
    is a continuous negative definite function on $\real^d$. If $p\in (1,2]$, the L\'evy measure has full support.
\end{lemma}
It is interesting to note that the Minkowski distances for $p>2$ and $d\geq
2$ are never negative definite functions. This is the consequence of
Schoenberg's problem, cf. Zastavnyi \cite[p.~56, Eq.~(3)]{zast00}.
\begin{proof}[Proof of Lemma~\ref{lem:p-minkowski}]
    Since each $x_i \mapsto |x_i|^p$, with $p\in [1,2]$, is a one-dimensional continuous negative definite function, we can use the formula \eqref{eq:lkf-sy} to see that
\begin{gather*}
\ell_p^p(x) = \begin{cases}
        \displaystyle
        \int_{\real^d} (1-\cos\scalp xr)\,\sum_{i=1}^d \frac{c_p\,\ddx_i}{|x_i|^{1+p}} \otimes\delta_0(\ddx_{(i)}), &\text{if \ } p\in [1,2),\\
        \displaystyle
        \scalp{x}{x}, &\text{if \ } p=2,
        \end{cases}
\end{gather*}
    where $x_{(i)} = (x_1,\dots,x_{i-1},x_{i+1},\dots x_d)\in\real^{d-1}$ and $c_p = \frac{p2^{p-1}\varGamma (\frac{p+1}{2} )}{\pi^{1/2}\varGamma (1-\frac{p}2 )}$ is the constant of the one-dimensional $p$-stable L\'evy measure, cf. Table~\ref{fig-cndf}.

    This means that $\ell_p^p$ is itself a continuous negative definite function, but its L\'evy measure is concentrated on the coordinate axes. Writing $\ell_p(x) = f_p(\ell_p^p(x))$ with
\begin{gather*}
f_p(\tau) = \tau^{1/p} = \gamma_{1/p}\int
_0^\infty \bigl(1-e^{-\tau t}\bigr)\,
\frac{\ddt}{t^{1+1/p}}, \quad \tfrac{1}p\in \bigl[\tfrac{1}2,1 \bigr],\;
\gamma_{1/p} = \frac{1}{p\varGamma (1-\frac{1}p )},
\end{gather*}
    shows that $\ell_p$ can be represented as a combination of the Bernstein function $f_p$ and the negative definite function $\ell_p^p$. In other words, $\ell_p$ is subordinate to $\ell_p^p$ in the sense of Bochner (cf. Sato~\cite[Chap.~30]{sato} or \cite[Chap.~5, Chap.~13.1]{ssv}) and it is possible to find the corresponding L\'evy--Khintchine representation, cf.~\cite[Thm.~30.1]{sato}. We have
\begin{gather*}
\ell_p(x) = \begin{cases}
        \displaystyle
        \int_{\real^d} (1-\cos\scalp xr)\,\sum_{i=1}^d \frac{c_p\,\ddx_i}{|x_i|^{2}} \otimes\delta_0(\ddx_{(i)}), &\text{if \ } p=1,\\
        \displaystyle
        \int_{\real^d} (1-\cos\scalp xr)\,\int_0^\infty \prod_{i=1}^d g_t(x_i)\,\frac{\ddt}{t^{1+1/p}}, &\text{if \ } p\in (1,2),\\
        \displaystyle
        \sqrt{\scalp{x}{x}}, &\text{if \ } p=2,
        \end{cases}
\end{gather*}
    where $x_i\mapsto g_t(x_i)$ is the probability density of the random variable $t^{p}X$ where $X$ is a one-dimensional, symmetric $1/p$-stable random variable.

    Although the $1/p$-stable density is known explicitly only for $1/p \in \{1,2\}$, one can show~-- this follows, e.g. from \cite[Thm.~15.10]{sato}~-- that it is strictly positive, i.e. the L\'evy measure of $\ell_p$, $p\in (1,2)$ has full support. For $p=1$ the measure does not have full support, since it is concentrated on the coordinate axes. For $p=2$, note that $\ell_2(x) = |x|$ corresponds to the Cauchy distribution with L\'evy measure given in Table~\ref{fig-cndf}, which has full support.
\end{proof}

Using the L\'evy--Khintchine representation \eqref{eq:lkf-sy} it is not hard
to see,\break cf.~\cite[Lem.~3.6.21]{Jaco2001}, that square roots of real-valued
cndfs are subadditive, i.e.
\begin{equation}
\label{ar-e10} \sqrt{\varTheta(x+y)} \leq \sqrt{\varTheta(x)}+\sqrt{
\varTheta(y)},\quad x,y\in\real^d
\end{equation}
and, consequently,
\begin{equation}
\label{ar-e11} \varTheta(x+y) \leq 2 \bigl(\varTheta(x)+\varTheta(y) \bigr),\quad
x,y\in\real^d.
\end{equation}

Using a standard argument, e.g. \cite[p.~44]{BoetSchiWang2013}, we can
derive from \eqref{ar-e10}, \eqref{ar-e11} that cndfs grow at most
quadratically as $x\to\infty$,
\begin{equation}
\label{ar-e12} \varTheta(x) \leq 2\sup_{|y|\leq 1}\varTheta(y)
\bigl(1+|x|^2\bigr).
\end{equation}

We will assume that $\varTheta(0)=0$ is the only zero of the function
$\varTheta$~-- incidentally, this means that $x\mapsto \ee^{-\varTheta(x)}$
is the characteristic function of a(n infinitely divisible) random variable
the distribution of which is non-lattice.
This and \eqref{ar-e10} show that
$(x,y)\mapsto\sqrt{\varTheta(x-y)}$ is a metric on $\real^d$ and
$(x,y)\mapsto \varTheta(x-y)$ is a quasi-metric, i.e. a function which
enjoys all properties of a metric, but the triangle inequality holds with a
multiplicative constant $c>1$. Metric measure spaces of this type have been
investigated by Jacob et al. \cite{JacoKnopLandSchi2011}. Historically, the
notion of negative definiteness has been introduced by I.J.~Schoenberg
\cite{Scho1938a} in a geometric context: he observed that for a real-valued
cndf\querymark{Q1} $\varTheta$ the function $d_\varTheta(x,y) := \sqrt{\varTheta(x-y)}$ is a metric on $\real^d$ and
that these are the only metrics such that $(\real^d,d_\varTheta)$ can be
isometrically embedded into a Hilbert space. In other words: $d_\varTheta$
behaves like a standard Euclidean metric in a possibly infinite-dimensional
space.

\subsection{Measuring independence of random variables with metrics}\label{ar-miwm}

Let $X,Y$ be random variables with values in $\real^m$ and $\real^n$,
respectively, and write $\law(X)$ and $\law(Y)$ for the corresponding
probability laws. For any metric $d(\cdot,\cdot)$ defined on the family of
$(m+n)$-dimensional probability distributions we have
\begin{equation}
\label{eq:inde-by-metric} d \bigl(\law(X,Y),\law(X)\otimes \law(Y) \bigr) = 0 \quad
\text{if, and only if, $X$, $Y$ are independent}.
\end{equation}

This equivalence can obviously be extended to finitely many random variables
$X_i$, $i=1,\dots,n$, taking values in $\real^{d_i}$, respectively: Set
$d:=d_1+\cdots+d_n$, take any metric $d(\cdot,\cdot)$ on the $d$-dimensional
probability distributions and consider $d (\law(X_1,\dots,X_n),
\bigotimes_{i=1}^n\law(X_i) )$. Moreover, the random variables $X_i,
i=1,\dots,n$, are independent if, and only if, $(X_1,\dots,X_{k-1})$ and
$X_k$ are independent for all $2\leq k\leq n$.\footnote{This is an immediate
consequence of the characterization of independence using characteristic
functions: $X$, $Y$ are independent if, and only if, $\Ee
\ee^{\ii\scalp{\xi}{X}+\ii\scalp{\eta}{Y}} = \Ee \ee^{\ii\scalp{\xi}{X}}\Ee
\ee^{\ii\scalp{\eta}{Y}}$ for all $\xi,\eta$.} In other words:
$X_1,\dots,X_n$ are independent if, and only if, for metrics on the
$d_1+\cdots+d_k$-dimensional probability distributions the distance of
$\law(X_1,\dots,X_k)$ and $\law(X_1,\dots,X_{k-1})\otimes\law(X_k)$ is zero
for $k=2,\ldots,n$. Thus, as in \eqref{eq:inde-by-metric}, only the concept
of independence of pairs of random variables is needed. In
\cite[Sec.~3.1]{part2} we use a variant of this idea to characterize
multivariate independence.

Thus \eqref{eq:inde-by-metric} is a good starting point for the construction
of (new) estimators for independence. For this it is crucial that the
(empirical) distance 
be computationally feasible. For discrete
distributions with finitely many values this yields the classical chi-squared
test of independence (using the $\chi^2$-distance). For more general
distributions other commonly used distances (e.g. relative entropy,
Hellinger distance, total variation, Prokhorov distance, Wasserstein
distance) might be employed (e.g. \cite{BerrSamw2017}), provided that they
are computationally feasible. It turns out that the latter\querymark{Q2} is, in particular,
satisfied by the following distance.
\begin{definition}\label{ar-05}
    Let $U,V$ be $d$-dimensional random variables and denote by $f_U$, $f_V$ their characteristic functions. For any symmetric measure $\rho$ on $\rdohne$ with full support we define the distance
\begin{equation}
\label{eq:d-general} d_\rho\bigl(\law(U),\law(V)\bigr) :=
\|f_U-f_V\|_{L^2(\rho)} = \biggl(\int
|f_U(r)-f_V(r)|^2 \,\rho(\ddr)
\biggr)^{1/2}.
\end{equation}
\end{definition}
The assumption that $\rho$ has full support, i.e.  $\rho(G)>0$ for every
nonempty open set $G\subset\rdohne$, ensures that $d_\rho(\law(U),\law(V)) =
0$ if, and only if, $\law(U) = \law(V)$, hence $d_\rho(\law(U),\law(V))$ is a
metric. The symmetry assumption on $\rho$ is not essential since the
integrand appearing in \eqref{eq:d-general} is even; therefore, we can always
replace $\rho(\ddr)$ by its symmetrization $\frac{1}2(\rho(\ddr)+\rho(-\ddr))$.

Currently it is unknown how the fact that the L\'evy measure $\rho$ has full
support can be expressed in terms of the cndf $\varTheta(u) = \int
(1-\cos\scalp ur)\,\rho(\ddr)$ given by $\rho$, see \eqref{eq:lkf-sy}.

Note that $d_\rho(\law(U),\law(V))$ is always well-defined in $[0,\infty]$.
Any of the following conditions ensure that $d_\rho(\law(U),\law(V))$ is
finite:
\begin{enumerate}
\item $\rho$ is a finite measure;
\item $\rho$ is a symmetric L\'evy measure (cf. \eqref{ar-e08}) and
    $\Ee|U|+\Ee|V|<\infty$.

Indeed, $d_\rho(\law(U),\law(V))<\infty$ follows from \xch{the}{the the} integrability
properties \eqref{ar-e08} of the L\'evy measure $\rho$ and the elementary
estimates
\begin{gather*}
|f_U(r)-f_V(r)| \leq \Ee |e^{\ii\scalp rU}-e^{\ii\scalp rV}|
\leq \Ee|\scalp{r} {(U-V)}|\leq |r|\cdot\Ee|U-V|
\\
\text{and}\quad |f_U(r)-f_V(r)|\leq 2.
\end{gather*}
\end{enumerate}

We obtain further sufficient conditions for $V(X,Y)<\infty$ in terms of
moments of the real-valued cndf $\varTheta$, see \eqref{eq:lkf-sy}, whose
L\'evy measure is $\rho$ and with $Q=0$.
\begin{proposition}\label{ar-31}
    Let $\rho$ be a symmetric L\'evy measure on $\rdohne$ with full support and denote by $\varTheta(u) = \int (1-\cos\scalp ur)\,\rho(\ddr)$, $u\in\real^d$, the real-valued cndf with L\'evy triplet $(l=0,Q=0,\rho)$. For all $d$-dimensional random variables $U,V$ the following assertions hold:
\begin{enumerate}
\item\label{ar-31-a} Assume that $(U',V')$ is an i.i.d.  copy of $(U,V)$.
    Then
\begin{equation}
\label{ar-e30} \Ee\varTheta\bigl(U-U'\bigr) + \Ee \varTheta
\bigl(V-V'\bigr) \leq 2\Ee\varTheta\bigl(U-V'\bigr).
\end{equation}

\item\label{ar-31-b} Let $U'$ be an i.i.d.  copy of $U$. Then
\begin{equation}
\label{eq:EPsi-new} \Ee\varTheta(U-V)\leq 2 \bigl(\Ee\varTheta(U)+\Ee\varTheta(V)
\bigr)
\end{equation}
    and for $V=U'$ one has $\Ee\varTheta(U-U')\leq 4\Ee\varTheta(U)$.

\item\label{ar-31-c} In $[0,\infty]$ we always have
\begin{equation}
\label{eq:v2bound-new} d_\rho^2\bigl(\law(U),\law(V)\bigr) \leq 4
\bigl(\Ee\varTheta(U) + \Ee\varTheta(V) \bigr).
\end{equation}

\item\label{ar-31-d} Let $(U',V')$ be an i.i.d.  copy of $(U,V)$ and
    assume  $\Ee\varTheta(U)+\Ee\varTheta(V)<\infty$. Then the following
    equality holds and all terms are finite
\begin{equation}
\label{ar-e36} d_\rho^2\bigl(\law(U),\law(V)\bigr) = 2\Ee
\varTheta\bigl(U-V'\bigr) - \Ee\varTheta\bigl(U-U'
\bigr) - \Ee\varTheta\bigl(V-V'\bigr).
\end{equation}
\end{enumerate}
\end{proposition}

\begin{proof}
    Let us assume first that $\rho$ is a finite L\'evy measure, thus, $\varTheta$ a bounded cndf. We denote by $U'$ an i.i.d.  copy of $U$. Since $\law(U-U')$ is symmetric, we can use Tonelli's theorem to get
\begin{align} \int \bigl( 1-f_U(r)
\overline{f_U(r)} \bigr)\rho(\ddr) &= \int \bigl(1 - \Ee\bigl(
\ee^{\ii \scalp{U}{r}} \ee^{-\ii \scalp{U'}{r}}\bigr) \bigr) \rho(\ddr)\notag
\\
&= \int \Ee \bigl(1 - \cos \bigl[\scalp{\bigl(U-U'\bigr)} {r}
\bigr] \bigr) \,\rho(\ddr)\notag
\\
&= \Ee\varTheta\bigl(U-U'\bigr)\notag
\\
&= \int \varTheta\bigl(u-u'\bigr)\, \Pp_U \otimes
\Pp_U\bigl(\ddu,\ddu'\bigr).  \label{eq:Etheta}
\end{align}

    Now we consider an i.i.d.  copy $(U',V')$ of $(U,V)$ and use the above equality in \eqref{eq:d-general}. This yields
\begin{align}
 &\notag d_\rho^2\bigl(\law(U),\law(V)\bigr)
\\
&\notag= \int \bigl( f_U(r)\overline{f_U(r)} -
{f_U(r)}\overline{f_V(r)} - f_V(r)
\overline{f_U(r)} + f_V(r)\overline{f_V(r)}
\bigr) \rho(\ddr)
\\
&\notag= \int \bigl( f_U(r)\overline{f_U(r)} - 1 + 2
- 2 \Re \bigl({f_U(r)}\overline{f_V(r)} \bigr) +
f_V(r)\overline{f_V(r)} - 1 \bigr) \rho(\ddr)
\\
&= 2\Ee\varTheta\bigl(U-V'\bigr) - \Ee\varTheta
\bigl(U-U'\bigr) - \Ee \varTheta\bigl(V-V'\bigr).\label{eq:d2-def}
\end{align}
    This proves \eqref{ar-e36} and, since $d_\rho(\law(U),\law(V))\geq 0$, also \eqref{ar-e30}. Combining Part~b) and \eqref{ar-e36} yields \eqref{eq:v2bound-new}, while \eqref{eq:EPsi-new} immediately follows from the subadditivity of a cndf \eqref{ar-e11}.

    If $\rho$ is an arbitrary L\'evy measure, its truncation $\rho_\epsilon(\ddr) := \I_{(\epsilon,\infty)}(|r|)\,\rho(\ddr)$ is a finite L\'evy measure and the corresponding cndf $\varTheta_\epsilon$ is bounded. In particular, we have \ref{ar-03-a}--\ref{ar-03-d} for $\rho_\epsilon$ and $\varTheta_\epsilon$. Using monotone convergence we get
\begin{gather*}
\varTheta(u) = \sup_{\epsilon>0} \varTheta_\epsilon(u) = \sup
_{\epsilon>0} \int_{|r|>\epsilon} (1-\cos\scalp{u} {r})\,
\rho(\ddr).
\end{gather*}

    Again by monotone convergence we see that the assertions a)--c) remain valid for general L\'evy measures~-- if we allow the expressions to attain values in $[0,\infty]$. Because of \eqref{eq:EPsi-new}, the moment condition assumed in Part~d) ensures that the limits
\begin{gather*}
\lim_{\epsilon\to 0}\Ee\varTheta\bigl(U-V'\bigr) = \sup
_{\epsilon>0}\Ee\varTheta\bigl(U-V'\bigr)\quad
\text{etc.}
\end{gather*}
    are finite, and \eqref{ar-e36} carries over to the general situation.
\end{proof}

\begin{remark}\label{ar-33}
  a)      Since $U$ and $V$ play symmetric roles in \eqref{eq:d2-def} it is clear that
\begin{align}
&d_\rho^2\bigl(\law(U),
\law(V)\bigr)
\notag\\
&= \Ee\varTheta\bigl(U-V'\bigr) + \Ee\varTheta
\bigl(V-U'\bigr) - \Ee\varTheta\bigl(U-U'\bigr) - \Ee
\varTheta\bigl(V-V'\bigr)
\notag\\
&= \int \varTheta\bigl(u-u'\bigr)\, (2\Pp_U \otimes
\Pp_V - \Pp_U \otimes \Pp_U -
\Pp_V \otimes \Pp_V) \bigl(\ddu,\ddu'
\bigr). \label{eq:d2-def-2}\end{align}

b)      While $d_\rho(\law(U),\law(V))\in [0,\infty]$ is always defined, the right-hand side of \eqref{ar-e36} needs attention. If we do not assume the moment condition $\Ee\varTheta(U)+\Ee\varTheta(V)<\infty$, we still have
\begin{equation}
\label{ar-e38} d_\rho^2\bigl(\law(U),\law(V)\bigr) = \lim
_{\epsilon\to 0} \xch{\bigl(2\Ee\varTheta_\epsilon\bigl(U-V'
\bigr)-\varTheta_\epsilon\bigl(U-U'\bigr)-
\varTheta_\epsilon\bigl(V-V'\bigr) \bigr),}{\bigl(2\Ee\varTheta_\epsilon\bigl(U-V'
\bigr)-\varTheta_\epsilon\bigl(U-U'\bigr)-
\varTheta_\epsilon\bigl(V-V'\bigr) \bigr)}
\end{equation}
    but it is not clear whether the limits exist for each term.

    The moment condition $\Ee\varTheta(U)+\Ee\varTheta(V)<\infty$ is sharp in the sense that it follows from $\Ee\varTheta(U-V')<\infty$: Since $U$ and $V'$ are independent, Tonelli's theorem entails that $\Ee\varTheta(u-V')<\infty$ for some $u\in\real^d$. Using the symmetry and sub-additivity of $\sqrt{\varTheta}$, see \eqref{ar-e11}, we get $\varTheta(V')\leq 2 (\varTheta(u-V')+\varTheta(u) )$, i.e. $\Ee\varTheta(V')<\infty$; $\Ee\varTheta(U)<\infty$ follows in a similar fashion.

c)      Since a cndf $\varTheta$ grows at most quadratically at infinity, see \eqref{ar-e12}, it is clear that
\begin{equation}
\label{ar-e40} \Ee \bigl(|U|^2 \bigr) + \Ee \bigl(|V|^2
\bigr) < \infty \quad\text{implies}\quad \Ee\varTheta(U)+\Ee\varTheta(V) < \infty.
\end{equation}
    One should compare this to the condition $\Ee|U|+\Ee|V|<\infty$ which ensures the finiteness of $d_\rho^2(\law(U),\law(V))$, but not necessarily the finiteness of the terms appearing in the representation \eqref{ar-e36}.

d)  As described at the beginning of this section,  a measure of independence of $X_1,\ldots,X_n$ is given by $d_\rho (\law(X_1,\dots,X_n), \bigotimes_{i=1}^n\law(X_i) )$. This can be estimated by empirical estimators for \eqref{ar-e36}. For the $1$-stable (i.e. Cauchy) cndf, see Table~1, this direct approach to (multivariate) independence has recently been proposed by \cite{JinMatt2017}~-- but the \emph{exact} estimators become computationally challenging even for small samples. A further approximation recovers a computationally feasible estimation, resulting in a loss of power compared with our approach, cf.~\cite{Boet2017}.
\end{remark}

It is worth mentioning that the metric $d_\rho$ can be used to describe
convergence in distribution.
\begin{lemma}\label{ar-35}
Let $\rho$ be a finite symmetric measure with full support, then $d_\rho$
given in \eqref{eq:d-general} is a metric which characterizes convergence in
distribution, i.e. for random variables $X_n$, $n\in\nat$, and $X$ one has
\begin{equation}
\label{ar-e50} X_n\xrightarrow[n\to\infty]{d} X \iff
d_\rho\bigl(\law(X_n),\law(X)\bigr) \xrightarrow[n\to
\infty]{} 0.
\end{equation}
\end{lemma}
The proof below shows that the implication ``$\Leftarrow$'' does not need the
finiteness of the L\'evy measure $\rho$.

\begin{proof}
    Convergence in distribution implies pointwise convergence of the characteristic functions. Therefore, we see by dominated convergence and because of the obvious estimates $|f_{X_n}|\leq 1$ and $|f_X|\leq 1$ that
\begin{align*}
\lim_{n\to\infty} \int_{\real^d}
|f_{X_n}(r)-f_X(r)|^2 \,\rho(\ddr) = \int
_{\real^d} \lim_{n\to\infty} |f_{X_n}(r)-f_X(r)|^2
\,\rho(\ddr) = 0.
\end{align*}

    Conversely, assume that $\lim_{n\to\infty} d_\rho(\law(X_n),\law(X))=0$. If we interpret this as convergence in $L^2(\rho)$, we see that there is a Lebesgue a.e. convergent subsequence $f_{X_{n(k)}}\to f_X$; since $f_{X_{n(k)}}$ and $f_X$ are characteristic functions, this convergence is already pointwise, hence locally uniform, see Sasv\'ari \cite[Thm.~1.5.2]{Sasv1994}. By L\'evy's continuity theorem, this entails the convergence in distribution of the corresponding random variables. Since the limit does not depend on the subsequence, the whole sequence must converge in distribution.
\end{proof}

\subsection{An elementary estimate for log-moments}

 Later on we need certain log-moments of the norm of a random vector. The following lemma allows us to formulate these moment conditions in terms of the coordinate processes.
\begin{lemma}\label{ar-61}
    Let $X,Y$ be one-dimensional random variables and $\epsilon>0$. Then\break $\Ee\log^{1+\epsilon} (1\vee \sqrt{X^2+Y^2} )$ is finite if, and only if, the moments $\Ee\log^{1+\epsilon}(1+X^2)$ and $\Ee\log^{1+\epsilon}(1+Y^2)$ are finite.
\end{lemma}
\begin{proof}
    Assume that $\Ee\log^{1+\epsilon}(1+X^2) + \Ee\log^{1+\epsilon}(1+Y^2)<\infty$. Since
\begin{gather*}
1\vee \sqrt{X^2+Y^2} = \sqrt{\bigl(X^2+Y^2
\bigr) \vee 1} \leq \xch{\sqrt{\bigl(1+X^2\bigr) \bigl(1+Y^2
\bigr)},}{\sqrt{\bigl(1+X^2\bigr) \bigl(1+Y^2
\bigr)}}
\end{gather*}
    we can use the elementary estimate $(a+b)^{1+\epsilon}\leq 2^\epsilon(a^{1+\epsilon}+b^{1+\epsilon})$, $a,b\geq 0$, to get
\begin{align*}
\Ee\log^{1+\epsilon} \bigl(1\vee \sqrt{X^2+Y^2} \bigr)
&\leq \Ee \biggl[ \biggl(\frac{1}2\log\bigl(1+X^2\bigr)+
\frac{1}2\log\bigl(1+Y^2\bigr) \biggr)^{1+\epsilon} \biggr]
\\
&\leq \frac{1}2 \Ee\log^{1+\epsilon}\bigl(1+X^2\bigr) +
\frac{1}2 \Ee\log^{1+\epsilon}\bigl(1+Y^2\bigr).
\end{align*}

\noindent
    Conversely, assume that $\Ee\log^{1+\epsilon} (1\vee \sqrt{X^2+Y^2} )<\infty$. Then we have
\begin{align*}
\Ee &\log^{1+\epsilon}\bigl(1+X^2\bigr)
\\
&= \Ee \bigl[ \I_{\{|X|<1\}} \log^{1+\epsilon}\bigl(1+X^2
\bigr) \bigr] 
 + \Ee \bigl[ \I_{\{|X|\geq 1\}} \log^{1+\epsilon}
\bigl(1+X^2\bigr) \bigr]
\\
&\leq \log^{1+\epsilon}2 + \Ee\log^{1+\epsilon} \bigl[\bigl(2X^2
\bigr)\vee 2 \bigr]
\\
&\leq \log^{1+\epsilon}2 + \Ee \bigl[ \bigl(\log 2 + \log\bigl(1\vee
X^2\bigr) \bigr)^{1+\epsilon} \bigr]
\\
&\leq \bigl(1+2^\epsilon\bigr)\log^{1+\epsilon}2 + \Ee
\log^{1+\epsilon}\bigl(1\vee \bigl(X^2+Y^2\bigr)\bigr)
\\
&\leq \bigl(1+2^\epsilon\bigr)\log^{1+\epsilon}2 + 2^{1+\epsilon}\Ee
\log^{1+\epsilon} \bigl(1\vee \sqrt{X^2+Y^2} \bigr),
\end{align*}
    and $\Ee\log^{1+\epsilon}(1+Y^2)<\infty$ follows similarly.
\end{proof}

\section{Generalized distance covariance}\label{sec:gdc}

Sz\'ekely et al. \cite{SzekRizzBaki2007,SzekRizz2009} introduced distance
covariance for two random variables $X$ and $Y$ with values in $\real^m$ and
$\real^n$ as
\begin{gather*}
\mathcal{V}^2(X,Y;w) := \int_{\real^n}\int
_{\real^m} \bigl\llvert f_{(X,Y)}(x,y)-f_X(x)f_Y(y)
\bigr\rrvert ^2\,w(x,y)\,\ddx\,\xch{\ddy,}{\ddy}
\end{gather*}
with the weight $w(x,y)=w_{\alpha,m}(x)w_{\alpha,n}(y)$ where
$w_{\alpha,m}(x)=c(p,\alpha) |x|^{-m-\alpha}$, $m,n\in\nat$, $\alpha\in
(0,2)$. It is well known from the study of infinitely divisible distributions
(see also Sz\'ekely \& Rizzo \cite{SzekRizz2005}) that $w_{\alpha,m}(x)$ is
the density of an $m$-dimensional $\alpha$-stable L\'evy measure, and the
corresponding cndf is just $|x|^\alpha$.

We are going to extend distance covariance to products of L\'evy measures.
\begin{definition}\label{gdc-05}
    Let $X$ and $Y$ be random variables with values in $\real^m$ and $\real^n$ and $\rho:=\mu\otimes\nu$ where $\mu$ and $\nu$ are symmetric L\'evy measures on $\rmohne$ and $\rnohne$, both having full support. The \emph{generalized distance covariance} $V(X,Y)$ is defined as
\begin{equation}
\label{eq:v2-def} V^2(X,Y) = \iint_{\real^{m+n}} |f_{(X,Y)}(s,t)-f_X(s)f_Y(t)|^2
\,\mu(\dds)\,\nu(\ddt).
\end{equation}
\end{definition}
By definition, $V(X,Y) = \|f_{(X,Y)}-f_X\otimes f_Y\|_{L^2(\rho)}$ and, in
view of the discussion in Section~\ref{ar-miwm}, we have
\begin{equation}
\label{eq:veqd} V(X,Y) = d_\rho\bigl(\law(U),\law(V)\bigr),
\end{equation}
where $U:=(X_1,Y_1), V:=(X_2,Y_3)$ and $(X_1,Y_1), (X_2,Y_2), (X_3,Y_3)$ are
i.i.d.  copies of $(X,Y)$.\footnote{This is a convenient way to say that
$\law(U) = \law((X,Y))$ while $\law(V) = \law(X)\otimes\law(Y)$.} It is clear
that the product measure $\rho$ inherits the properties ``symmetry'' and
``full support'' from its marginals $\mu$ and $\nu$.

From the discussion following Definition~\ref{ar-05} we immediately get the
next lemma.
\begin{lemma}\label{gdc-07}
    Let $V^2(X,Y)$ be generalized distance covariance of the $m$- resp. $n$-dimensional random variables $X$ and $Y$, cf. \eqref{eq:v2-def}. The random variables $X$ and $Y$ are independent if, and only if, $V^2(X,Y)=0$.
\end{lemma}

\subsection{Generalized distance covariance with finite L\'evy measures}\label{sec:gdc-finite}
Fix the dimensions $m,n\in\nat$, set $d:=m+n$, and assume that the measure
$\rho$ is of the form $\rho = \mu \otimes \nu$ where $\mu$ and $\nu$ are
finite symmetric L\'evy measures on $\rmohne$ and $\rnohne$, respectively. If
we integrate the elementary estimates
\begin{gather*}
1\wedge|s|^2 \leq 1\wedge\bigl(|s|^2+|t|^2
\bigr) \leq \bigl(1\wedge|s|^2\bigr) + \xch{\bigl(1\wedge |t^2|
\bigr),}{\bigl(1\wedge |t^2|
\bigr)}
\\
1\wedge|t|^2 \leq 1\wedge\bigl(|s|^2+|t|^2
\bigr) \leq \bigl(1\wedge|s|^2\bigr) + \xch{\bigl(1\wedge |t^2|\bigr),}{\bigl(1\wedge |t^2|\bigr)}
\end{gather*}
with respect to $\rho(\dds,\ddt)=\mu(\dds)\,\nu(\ddt)$, it follows that
$\rho$ is a L\'evy measure if $\mu$ and $\nu$ are finite L\'evy
measures.\footnote{This argument also shows that the product measure $\rho$
can only be a L\'evy measure, if the marginals are finite measures. In this
case, $\rho$ is itself a finite measure.} We also assume that $\mu$ and
$\nu$, hence $\rho$, have full support.

Since $V(X,Y)$ is a metric in the sense of Section~\ref{ar-miwm} we can use
all results from the previous section to derive various representations of
generalized distance covariance.

We write $\varPhi$, $\varPsi$ and $\varTheta$ for the bounded cndfs induced
by $\mu(\dds)$, $\nu(\ddt)$, and $\rho(\ddr)=\mu(\dds)\,\nu(\ddt)$,
\begin{align}
\notag\varPhi(x) =&{} \int_{\real^m} (1-
\cos\scalp{x} {s} ) \,\mu(\dds), \quad \varPsi(y) = \int_{\real^n}
(1- \cos\scalp{y} {t} ) \,\nu(\ddt)
\\*
&{}\text{and}\quad \varTheta(x,y) = \varTheta(u) = \int_{\real^{m+n}}
(1- \cos\scalp{u} {r} )\,\rho(\ddr), \label{eq:PsiPhi}
\end{align}
with $u=(x,y)\in\real^{m+n}$ and $r=(s,t)\in\real^{m+n}$. The symmetry in
each variable and the elementary identity
\begin{align*}
\frac{1}{2} \bigl[1-&\cos(\scalp{x} {s}-\scalp{y} {t}) + 1-\cos(\scalp{x}
{s}+\scalp{y} {t}) \bigr]
\\
&= (1- \cos\scalp{y} {t}) + (1-\cos \scalp{x} {s}) - (1-\cos\scalp{x} {s}) (1-
\cos\scalp{y} {t})
\end{align*}
yield the representation
\begin{equation}
\label{eq:theta-rel} \varTheta(x,y) = \varPsi(y)\mu\bigl(\real^m\bigr) +
\varPhi(x)\nu\bigl(\real^n\bigr) - \varPhi(x)\varPsi(y).
\end{equation}

We can now easily apply the results from Section~\ref{ar-miwm}. In order to
do so, we consider six i.i.d.  copies $(X_i,Y_i)$, $i=1,\dots,6$, of the
random vector $(X,Y)$, and set $U:=(X_1,Y_1), V:=(X_2,Y_3), U':=(X_4,Y_4),
V':=(X_5,Y_6)$. This is a convenient way to say that
\begin{gather*}
\law(U) = \law\bigl(U'\bigr) = \law\bigl((X,Y)\bigr) \et \law(V) =
\law\bigl(V'\bigr) = \law(X)\otimes\law(Y)
\end{gather*}
and $U,U',V$ and $V'$ are independent.\footnote{In other words: in an
expression of the form $f(X_i,X_j,Y_k,Y_l)$ all random variables are
independent if, and only if, all indices are different. As soon as two
indices coincide, we have (some kind of) dependence.}

The following formulae follow directly from
Proposition~\ref{ar-31}.\ref{ar-31-d} and Remark~\ref{ar-33}.a).
\begin{proposition}\label{gdc-08}
    Let $X$ and $Y$ be random variables with values in $\real^m$ and $\real^n$ and assume that $\mu, \nu$ are finite symmetric L\'evy measures on $\rmohne$ and $\rnohne$ with full support. Generalized distance covariance has the following representations
\begin{align}
V^2(X,Y) &= d_\rho^2\bigl(\law(U),\law(V)
\bigr)
\\
&=\label{eq:v2-ThetaUV} 2\Ee\varTheta\bigl(U-V'\bigr) - \Ee\varTheta
\bigl(U-U'\bigr) - \Ee \varTheta\bigl(V-V'\bigr)
\\
&=\notag 2\Ee\varTheta(X_1- X_5,Y_1-
Y_6) - \Ee\varTheta(X_1-X_4,Y_1-Y_4)
\\
 &\qquad\quad\mbox{} - \Ee \varTheta( X_2- X_5,
Y_3- Y_6)\label{eq:v2-ThetaXY}
\\
&= \Ee\varPhi(X_1-X_4)
\varPsi(Y_1-Y_4) + \Ee\varPhi(X_2-
X_5) \Ee\varPsi(Y_3- Y_6)\notag
\\
 &\qquad\quad\mbox{}- 2\Ee\varPhi(X_1-X_5)\label{eq:v2-PsiPhifator}
\varPsi(Y_1-Y_6).
\end{align}
\end{proposition}
The latter equality follows from \eqref{eq:theta-rel} since the terms
depending only on one of the variables cancel as the random variables
$(X_i,Y_i)$ are i.i.d. This gives rise to various further representations of
$V(X,Y)$.
\begin{corollary}\label{gdc-09}
    Let $(X,Y)$, $(X_i,Y_i)$, $\varPhi,\varPsi$ and $\mu,\nu$ be as in Proposition~\ref{gdc-08}. Generalized distance covariance has the following representations
\begin{align}
\notag &V^2(X,Y)
\\
&=\notag \Ee\varPhi(X_1-X_4)
\varPsi(Y_1-Y_4) - 2\Ee\varPhi(X_1-X_2)
\varPsi(Y_1-Y_3)
\\
\label{eq:v2-PsiPhi} &\qquad\quad\mbox{} + \Ee\varPhi(X_1- X_2) \Ee
\varPsi(Y_3- Y_4)
\\
&= \Ee \bigl[\varPhi(X_1-X_4)\cdot \bigl\{
\varPsi(Y_1-Y_4)-2\varPsi(Y_1-Y_3)+
\varPsi(Y_2-Y_3) \bigr\} \bigr]
\\
&=\label{eq:v2-Prod} \Ee \bigl[ \bigl\{\varPhi(X_1-X_4)-
\varPhi(X_4-X_2) \bigr\} \cdot \bigl\{
\varPsi(Y_4-Y_1)-\varPsi(Y_1-Y_3)
\bigr\} \bigr]
\\
&=\notag \Ee \bigl[ \bigl\{\varPhi(X_1-X_4)-\Ee
\bigl(\varPhi(X_4-X_1) \mid X_4\bigr) \bigr
\}
\\
&\qquad\quad\mbox{}\cdot \bigl\{\varPsi(Y_4-Y_1)-
\Ee\bigl(\varPsi(Y_1-Y_4) \mid Y_1\bigr)
\bigr\} \bigr].\label{eq:v2-cond}
\end{align}
\end{corollary}

Corollary~\ref{gdc-09} shows, in particular, that $V(X,Y)$ can be written as
a function of $(X_i,Y_i)$, $i=1,\dots, 4$,
\begin{equation}
\label{eq:v2-g} V^2(X,Y) = \Ee \bigl[g \bigl((X_1,Y_1),
\dots,(X_4,Y_4) \bigr) \bigr]
\end{equation}
for an appropriate function $g$; for instance, the formula
\eqref{eq:v2-PsiPhi} follows for
\begin{align*}
 &g\bigl((x_1,y_1),\dots,(x_4,y_4)
\bigr)
\\
&= \varPhi(x_1-x_4)\varPsi(y_1-y_4)-2
\varPhi(x_1-x_2)\varPsi(y_1-y_3)
+ \varPhi(x_1- x_2) \varPsi(y_3-
y_4).
\end{align*}

\begin{corollary}\label{gdc-11}
    Let $(X,Y)$, $(X_i,Y_i)$, $\varPhi,\varPsi$ and $\mu,\nu$ be as in Proposition~\ref{gdc-08} and write
\begin{align*}
\doverline{\varPhi} &= \varPhi(X_1-X_4)- \Ee\bigl(
\varPhi(X_4-X_1) \mid X_4\bigr)
\\
&\quad\qquad\mbox{} - \Ee\bigl(\varPhi(X_4-X_1) \mid
X_1\bigr) + \Ee \varPhi\xch{(X_1-X_4),}{(X_1-X_4)}
\\
\doverline{\varPsi} &= \varPsi(Y_1-Y_4)- \Ee\bigl(
\varPsi(Y_4-Y_1) \mid Y_4\bigr)
\\
&\quad\qquad\mbox{} - \Ee\bigl(\varPsi(Y_4-Y_1) \mid
Y_1\bigr) + \Ee \varPsi\xch{(Y_1-Y_4),}{(Y_1-Y_4).}
\end{align*}
    for the ``doubly centered'' versions of $\varPhi(X_1-X_4)$ and $\varPsi(Y_1-Y_4)$.
    Generalized distance covariance has the following representation
\begin{equation}
\label{eq:Vdoublecentered} V^2(X,Y) = \Ee [\,\doverline{\varPhi}\cdot\doverline{
\varPsi}\, ].
\end{equation}
\end{corollary}
\begin{proof}
Denote by $\overline{\varPhi} = \varPhi(X_1-X_4)-\Ee(\varPhi(X_4-X_1) \mid
X_4)$ and $\overline{\varPsi} = \varPsi(Y_1-Y_4)-\Ee(\varPsi(Y_4-Y_1) \mid
Y_1)$ the centered random variables appearing in \eqref{eq:v2-cond}. Clearly,
\begin{gather*}
\doverline\varPhi = \overline\varPhi - \Ee(\overline\varPhi \mid
X_1) \quad\text{and}\quad \doverline\varPsi = \overline\varPsi - \Ee(
\overline\varPsi \mid Y_4).
\end{gather*}
Thus,
\begin{align*}
\Ee [\doverline\varPhi\cdot\doverline\varPsi ] &= \Ee \bigl[ \bigl(\overline
\varPhi - \Ee(\overline\varPhi \mid X_1) \bigr) \cdot \bigl(\overline
\varPsi - \Ee(\overline\varPsi \mid Y_4) \bigr) \bigr]
\\
&= \Ee [\overline\varPhi\cdot\overline\varPsi ] + \Ee \bigl[\Ee(\overline\varPhi
\mid X_1)\cdot\Ee(\overline\varPsi \mid Y_4) \bigr]
\\
&\qquad\mbox{} - \Ee \bigl[\Ee(\overline\varPhi \mid X_1)\cdot
\overline\varPsi \bigr] - \Ee \bigl[\overline\varPhi\cdot\Ee(\overline\varPsi \mid
Y_4) \bigr].
\end{align*}
Since $X_1$ and $Y_4$ are independent, we have
\begin{gather*}
\Ee \bigl[\Ee(\overline\varPhi \mid X_1)\cdot\Ee(\overline\varPsi
\mid Y_4) \bigr] = \Ee \bigl[\Ee(\overline\varPhi \mid
X_1) \bigr]\cdot\Ee \bigl[\Ee(\overline\varPsi \mid Y_4)
\bigr] = \Ee [\overline\varPhi ]\cdot\Ee [\overline\varPsi ] = 0.
\end{gather*}
Using the tower property and the independence of $(X_1,Y_1)$ and $Y_4$ we get
\begin{align*}
\Ee \bigl[\Ee(\overline\varPhi \mid X_1)\cdot\overline\varPsi \bigr]
&= \Ee \bigl[ \Ee \bigl[\Ee(\overline\varPhi \mid X_1) \mid
Y_1,Y_4 \bigr] \cdot\overline\varPsi \bigr]
\\
&= \Ee \bigl[ \Ee \bigl[\Ee(\overline\varPhi \mid X_1) \mid
Y_1 \bigr] \cdot\overline\varPsi \bigr] = \xch{0,}{0}
\end{align*}
where we use, for the last equality, that $\overline\varPsi$ is orthogonal to
the $L^2$-space of $Y_1$-measurable functions. In a similar fashion we see
$\Ee [\overline\varPhi\cdot\Ee(\overline\varPsi \mid Y_4) ]=0$, and the
assertion follows because of \eqref{eq:v2-cond}.
\end{proof}

In Section \ref{sec:estimation} we will encounter further representations of
the generalized distance covariance if $X$ and $Y$ have discrete
distributions with finitely many values, as it is the case for empirical
distributions.

\subsection{Generalized distance covariance with arbitrary L\'evy measures}\label{sec:gdc-infinite}
So far, we have been considering finite L\'evy measures $\mu,\nu$ and bounded
cndfs $\varPhi$ and $\varPsi$ \eqref{eq:PsiPhi}. We will now extend our
considerations to products of unbounded L\'evy measures. The measure $\rho :=
\mu\otimes\nu$ satisfies the integrability condition
\begin{equation}
\label{gdc-e22} \iint\limits
_{\real^{m+n}} \!\!\bigl(1\wedge|x|^2\bigr) \bigl(1
\wedge|y|^2\bigr)\,\rho(\ddx,\ddy) = \!\!\int\limits
_{\real^m} \!\!\bigl(1
\wedge|x|^2\bigr)\,\mu(\ddx) \!\!\int\limits
_{\real^n} \!\! \bigl(1
\wedge|y|^2\bigr)\,\nu(\ddy) < \infty.
\end{equation}
Other than in the case of finite marginals, $\rho$ is no longer a L\'evy
measure, see the footnote on page~\pageref{sec:gdc-finite}. Thus, the
function $\varTheta$ defined in \eqref{eq:PsiPhi} need not be a cndf and we
cannot directly apply Proposition~\ref{ar-31}; instead we need the following
result ensuring the finiteness of $V(X,Y)$.
\begin{lemma} \label{thm:v2bound}
    Let $X, X'$ be i.i.d.  random variables on $\real^m$ and $Y, Y'$ be i.i.d.  random variables on $\real^n$; $\mu$ and $\nu$ are symmetric L\'evy measures on $\real^m$ and $\real^n$ with full support and with corresponding cndfs $\varPhi$ and $\varPsi$ as in \eqref{eq:PsiPhi}. Then
\begin{equation}
\label{eq:v2bound} V^2(X,Y) \leq \Ee\varPhi\bigl(X-X'
\bigr) \cdot \Ee\varPsi\bigl(Y-Y'\bigr) \leq 16\, \Ee\varPhi(X)
\cdot \Ee\varPsi(Y).
\end{equation}
\end{lemma}
\begin{proof}
    Following Sz\'ekely et al.~\cite[p.~2772]{SzekRizzBaki2007} we get
\begin{align}
\notag |f_{X,Y}(s,t) - f_X(s)f_Y(t)|^2
&= \bigl\llvert \Ee \bigl(\bigl(\ee^{\ii \scalp{X}{s}}-f_X(s)\bigr)
\bigl(\ee^{\ii \scalp{Y}{t}}-f_Y(t)\bigr) \bigr) \bigr\rrvert
^2
\\
\notag &\leq \bigl[\Ee \bigl(|\ee^{\ii \scalp{X}{s}}-f_X(s)||
\ee^{\ii \scalp{Y}{t}}-f_Y(t)| \bigr) \bigr]^2
\\
 &\leq \Ee \bigl[|\ee^{\ii \scalp{X}{s}}-f_X(s)|^2
\bigr] \cdot \Ee \bigl[|\ee^{\ii \scalp{Y}{t}}-f_Y(t)|^2
\bigr]\label{gdc-e24}
\end{align}
    and
\begin{align} \notag
\Ee \bigl[|\ee^{\ii \scalp{X}{s}}-f_X(s)|^2
\bigr] &= \Ee \bigl[\bigl(\ee^{\ii sX}-f_X(s)\bigr) \bigl(
\ee^{-\ii sX}-\overline{f_X(s)}\bigr) \bigr]
\\
&= 1- |f_X(s)|^2 = 1-f_X(s)
\overline{f_{X}(s)}. \end{align}
    Using \eqref{eq:Etheta} for $\rho = \mu$, $\varTheta=\varPhi$, $U=X$ and \eqref{eq:EPsi-new} for $\varTheta=\varPhi$, $U=X$, $V=X'$ shows
\begin{gather*}
\int_{\real^m}\Ee \bigl[|\ee^{\ii \scalp{X}{s}}-f_X(s)|^2
\bigr]\,\mu(\dds) = \Ee\varPhi\bigl(X-X'\bigr) \leq 4\Ee\varPhi(X),
\end{gather*}
    and an analogous argument for $\nu$ and $Y$ yields the bound \eqref{eq:v2bound}.
\end{proof}

Looking at the various representations
\eqref{eq:v2-PsiPhifator}--\eqref{eq:v2-cond} of $V(X,Y)$ it is clear that
these make sense as soon as all expectations in these expressions are finite,
i.e. some moment condition in terms of $\varPhi$ and $\varPsi$ should be
enough to ensure the finiteness of $V(X,Y)$ and all terms appearing in the
respective representations.

In order to use the results of the previous section we fix $\epsilon>0$ and
consider the \emph{finite} symmetric L\'evy measures
\begin{equation}
\label{eq:munu-trunc} \mu_\epsilon(\dds):= \frac{|s|^2}{\epsilon^2+|s|^2}\,\mu(\dds) \et
\nu_\epsilon(\ddt):= \frac{|t|^2}{\epsilon^2+|t|^2}\,\nu(\ddt),
\end{equation}
and the corresponding cndfs $\varPhi_\epsilon$ and $\varPsi_\epsilon$ given
by \eqref{eq:PsiPhi}; the product measure $\rho_\epsilon :=
\mu_\epsilon\otimes\nu_\epsilon$ is a finite L\'evy measure and the
corresponding cndf $\varTheta_\epsilon$ is also bounded (it can be expressed
by $\varPhi_\epsilon$ and $\varPsi_\epsilon$ through the formula
\eqref{eq:theta-rel}).

This allows us to derive the representations
\eqref{eq:v2-PsiPhifator}--\eqref{eq:v2-cond} for each $\epsilon>0$ and with
$\varPhi_\epsilon$ and $\varPsi_\epsilon$. Since we have $\varPhi =
\sup_{\epsilon}\varPhi_\epsilon$ and $\varPsi =
\sup_{\epsilon}\varPsi_\epsilon$, we can use monotone convergence to get the
representations for the cndfs $\varPhi$ and $\varPsi$ with L\'evy measures
$\mu$ and $\nu$, respectively. Of course, this requires the existence of
certain (mixed) $\varPhi$-$\varPsi$ moments of the random variables $(X,Y)$.

\begin{theorem}\label{gdc-21}
    Let $\mu$ and $\nu$ be symmetric L\'evy measures on $\rmohne$ and $\rnohne$ with full support and corresponding cndfs $\varPhi$ and $\varPsi$ given by \eqref{eq:PsiPhi}. For any random vector $(X,Y)$ with values in $\real^{m+n}$ satisfying the following moment condition
\begin{equation}
\label{eq:moments1} \Ee\varPhi(X) + \Ee\varPsi(Y) < \xch{\infty,}{\infty}
\end{equation}
    the generalized distance correlation $V(X,Y)$ is finite. If additionally
\begin{equation}
\label{eq:moments2} \Ee \bigl(\varPhi(X)\varPsi(Y) \bigr) < \infty,
\end{equation}
     then also the representations \eqref{eq:v2-PsiPhifator}--\eqref{eq:v2-cond} hold with all terms finite.\footnote{As before, we denote in these formulae by $(X_i,Y_i)$, $i=1,\dots,6$, i.i.d.  copies of $(X,Y)$.}
\end{theorem}
\begin{proof}
    We only have to check the finiteness. Using Lemma~\ref{thm:v2bound}, we see that \eqref{eq:moments1} guarantees that $V(X,Y)<\infty$.
    The finiteness of (all the terms appearing in) the representations \eqref{eq:v2-PsiPhifator}--\eqref{eq:v2-cond} follows from the monotone convergence argument, since the moment condition \eqref{eq:moments2} ensures the finiteness of the limiting expectations.
\end{proof}

\begin{remark}\label{gdc-23}
  a)      Using the H\"{o}lder inequality the following condition implies \eqref{eq:moments1} and \eqref{eq:moments2}:
\begin{equation}
\label{eq:moments-2} \Ee\varPhi^p(X) + \Ee\varPsi^q(Y) < \infty
\quad\text{for some $p,q>1$ with $\frac{1}p+\frac{1}q = 1$}.
\end{equation}
    If one of $\varPsi$ or $\varPhi$ is bounded then \eqref{eq:moments1} implies \eqref{eq:moments2}, and if both are bounded then the expectations are trivially finite.

    Since continuous negative definite functions grow at most quadratically, we see that  \eqref{eq:moments1} and \eqref{eq:moments2} also follow if
\begin{equation}
\label{eq:moments-3} \Ee |X|^{2p} + \Ee|Y|^{2q} < \infty \quad
\text{for some $p,q>1$ with $\frac{1}p+\frac{1}q = 1$}.
\end{equation}

b)      A slightly different set-up was employed by Lyons \cite{Lyon2013}: If the cndfs $\varPhi$ and $\varPsi$ are subadditive, then the expectation in \eqref{eq:v2-Prod} is finite. This is a serious restriction on the class of cndfs since subadditivity means that $\varPhi$ and $\varPsi$ can grow at most linearly at infinity, whereas general cndfs grow at most quadratically, cf.~\eqref{ar-e12}. Note, however, that square roots of real cndfs are always subadditive, cf.~\eqref{ar-e10}.
\end{remark}

\section{Estimating generalized distance covariance}\label{sec:estimation}

Let $(x_i,y_i)_{i=1,\dots,N}$ be a sample of $(X,Y)$ and denote by $(\widehat
X^{(N)},\widehat Y^{(N)})$ the random variable which has the corresponding
empirical distribution, i.e. the uniform distribution on
$(x_i,y_i)_{i=1,\dots,N}$; to be precise, repeated points will have the
corresponding multiple weight. By definition, $(\widehat X^{(N)},\widehat
Y^{(N)})$ is a bounded random variable and for i.i.d.  copies with
$\law((\widehat X_i^{(N)},\widehat Y_i^{(N)})) = \law((\widehat
X^{(N)},\widehat Y^{(N)}))$ for $i=1,\ldots,4$, we have using \eqref{eq:v2-g}
\begin{align}
\notag\Ee& \bigl[g \bigl(\bigl(\widehat X_1^{(N)},
\widehat Y_1^{(N)}\bigr),\dots, \bigl(\widehat
X_4^{(N)},\widehat Y_4^{(N)}\bigr)
\bigr) \bigr]
\\
&= \int_\rnohne\int_\rmohne \bigl
\llvert f_{(\widehat X^{(N)},\widehat Y^{(N)})}(s,t)-f_{\widehat X^{(N)}}(s)f_{\widehat Y^{(N)}}(t) \bigr
\rrvert ^2 \,\mu(\dds)\,\nu(\ddt).\label{est-e04}
\end{align}
The formulae \eqref{eq:v2-ThetaUV}--\eqref{eq:v2-cond} hold, in particular,
for the empirical random variables\break $(\widehat X^{(N)},\widehat Y^{(N)})$, and
so we get 
\begin{align}
&\Ee \bigl[g \bigl(\bigl(\widehat X_1^{(N)},\widehat
Y_1^{(N)}\bigr),\dots, \bigl(\widehat X_4^{(N)},
\widehat Y_4^{(N)}\bigr) \bigr) \bigr]\notag
\\
&\label{eq:vstat}= \frac{1}{N^4} \sum_{i,j,k,l = 1}^N
g \bigl((x_i,y_i),(x_j,y_j),(x_k,y_k),(x_l,y_l)
\bigr)
\\
&=\notag \frac{1}{N^4} \smash[b]{\sum_{i,j,k,l = 1}^N}
\bigl[\varPhi(x_i-x_k)\varPsi(y_i-y_k)
- 2 \varPhi(x_i-x_j)\varPsi(y_i-y_l)
\\
\label{eq:v2n} &\phantom{= \frac{1}{N^4} \sum\sum  [} + \varPhi(x_i-x_j)
\varPsi(y_k-y_l) \bigr]
\\
&=\notag\frac{1}{N^2} \sum_{i,k = 1}^N
\varPhi(x_i-x_k)\varPsi(y_i-y_k)
-\frac{2}{N^3} \sum_{i,j,l = 1}^N
\varPhi(x_i-x_j)\varPsi(y_i-y_l)
\\
&\label{eq:vpsistat}\qquad\mbox{}+ \frac{1}{N^2} \sum_{i,j = 1}^N
\varPhi(x_i-x_j) \frac{1}{N^2} \sum
_{k,l = 1}^N \varPsi(y_k-y_l)
\\
&=\notag \biggl(\frac{1}{N^2}-\frac{2}{N^3}+\frac{2}{N^4}
\biggr) \sum_{\substack{i,k = 1\\\text{distinct}}}^N
\varPhi(x_i-x_k)\varPsi(y_i-y_k)
\\
&\notag\qquad\mbox{} + \biggl( \frac{4}{N^4}- \frac{2}{N^3} \biggr)
\sum_{\substack{i,j,l = 1\\\text{distinct}}}^N \varPhi(x_i-x_j)
\varPsi(y_i-y_l)
\\
&\label{eq:ustat}\qquad\mbox{}+ \frac{1}{N^4} \sum_{\substack{i,j,k,l = 1\\\text{distinct}}}^N
\varPhi(x_i-x_j) \varPsi(y_k-y_l).
\end{align}
The sum in \eqref{eq:vstat} is~-- for functions $g$ which are symmetric under
permutations of their variables~-- a V-statistic.

\begin{definition}\label{est-03}
    The estimator $\VN^2 := \VN^2((x_1,y_1),\dots,(x_N,y_N))$ of $V^2(X,Y)$ is defined by \eqref{eq:v2n}.
\end{definition}
In abuse of notation, we also write $\VN^2((x_1,\dots,x_N),(y_1,\dots,y_N))$
and\break $\VN^2(\boldsymbol{x}, \boldsymbol{y})$ with $\boldsymbol{x} = (x_1,\dots,x_N)$, $\boldsymbol{y} =
(y_1,\dots,y_N)$ instead of the precise\break $\VN^2((x_1,y_1),\dots,(x_N,y_N))$.

Since all random variables are bounded, we could use any of the
representations \eqref{eq:vstat}--\eqref{eq:ustat} to define an estimator.  A
computationally feasible representation is given in the following lemma.
\begin{lemma}\label{est-05}
    The estimator $\VN^2$ of $V^2(X,Y)$ has the following representation using matrix notation:
\begin{equation}
\label{est-e06} \VN^2\bigl((x_1,y_1),
\dots,(x_N,y_N)\bigr) =\frac{1}{N^2} \trace
\bigl(B^{\top}A\bigr) = \frac{1}{N^2} \sum
_{k,l=1}^N A_{kl}\xch{B_{kl},}{B_{kl}}
\end{equation}
    where
\begin{equation}
\label{eq:distancematrices}
\begin{aligned}
A &= C^{\top} a C, \quad a = \bigl(\varPhi(x_k-x_l)\bigr)_{k,l=1,..,N},\\
B &= C^{\top} b C, \quad b = \bigl(\varPsi(y_k-y_l)\bigr)_{k,l=1,..,N},
\end{aligned}
\end{equation}
    and $C= I- \frac{1}{N}\I$ with $\I = (1)_{k,l=1,\dots,N}$ and $I = (\delta_{jk})_{j,k=1,\dots,N}$.
\end{lemma}

\begin{remark}\label{est-07}
    If $\varPhi(x) = |x|$ and $\varPsi(y)=|y|$ the matrices $a$ and $b$ in \eqref{eq:distancematrices} are Euclidean distance matrices. For general cndfs $\varPhi$ and $\varPsi$, the matrices $-a$ and $-b$ are conditionally positive definite (see Theorem~\ref{ar-03}), and $A$, $B$ are positive definite. This gives a simple explanation as to why the right-hand side of \eqref{est-e06} is positive.
\end{remark}

\begin{proof}[Proof of Lemma~\ref{est-05}] By definition,
\begin{align}
\trace\bigl(b^{\top}a\bigr) &= \sum_{i,j=1}^N a_{ij}b_{ij} = \sum_{i,j=1}^N \varPhi(x_i-x_j) \varPsi(y_i-y_j),\\
\trace\bigl(b^{\top}\I a\bigr) &= \sum_{i,j,k=1}^N a_{ij}b_{ik} = \sum_{i,j,k=1}^N \varPhi(x_i-x_j) \varPsi(y_i-y_k), \\
\trace\bigl(\I b^{\top}\I a\bigr) &= \sum_{i,j,k,l=1}^N a_{ij}b_{kl} = \sum_{i,j=1}^N \varPhi(x_i-x_j) \sum_{k,l=1}^N \varPsi(y_k-y_l),
\end{align}
and this allows us to rewrite \eqref{eq:vpsistat} as
\begin{equation}
\label{eq:vntrace} \VN^2 
 = \frac{1}{N^2} \trace
\bigl(b^{\top}a\bigr) - \frac{2}{N^3}\trace\bigl(b^{\top} \I
a\bigr) + \frac{1}{N^4} \trace\bigl(\I b^{\top} \I a\bigr).
\end{equation}
Observe that $C=C^\top$ and $CC^\top=C$. Using this and the fact that the
trace is invariant under cyclic permutations, we get
\begin{align*}
 \trace\bigl(B^{\top}A\bigr) &=\trace\bigl(C^{\top}b^{\top}CC^{\top}aC
\bigr)
\\
&=\trace\bigl(CC^{\top}b^{\top}CC^{\top}a\bigr) =\trace
\bigl(Cb^{\top}Ca\bigr).
\end{align*}
Plugging in the definition of $C$ now gives\vadjust{\eject}
\begin{align*}
\trace\bigl(&B^{\top}A\bigr) =\trace \bigl(\bigl(b^{\top}-
\tfrac{1}{N}\I b^{\top}\bigr) \bigl(a-\tfrac{1}{N}\I a\bigr)
\bigr)
\\
&=\trace\bigl(b^{\top}a\bigr) - \tfrac{1}{N}\trace\bigl(\I
b^{\top} a\bigr) - \tfrac{1}{N}\trace\bigl(b^{\top}\I a
\bigr) + \tfrac{1}{N^2} \trace\bigl(\I b^{\top}\I a\bigr)
\\
&=\trace\bigl(b^{\top}a\bigr) - \tfrac{2}{N}\trace
\bigl(b^{\top} \I a\bigr) + \tfrac{1}{N^2} \trace\bigl(\I
b^{\top} \I a\bigr).
\end{align*}
For the last equality we use that $a$ and $b$ are symmetric matrices.
\end{proof}

We will now show that $\VN^2$ is a consistent estimator for $V^2(X,Y)$.
\begin{theorem}[Consistency] \label{thm:VNconsistent} 
    Let $(X_i,Y_i)_{i=1,\dots, N}$ be i.i.d.  copies of $(X,Y)$ and write $\boldsymbol{X} := (X_1,\dots,X_N)$, $\boldsymbol{Y} := (Y_1,\dots, Y_N)$. If $\Ee [\varPhi(X)+\varPsi(Y) ]<\infty$  holds, then
\begin{equation}
\VN^2(\boldsymbol{X},\boldsymbol{Y}) \xrightarrow[N\to\infty]{} V^2(X,Y)
\quad \text{a.s.}
\end{equation}
\end{theorem}
\begin{proof}
    The moment condition $\Ee [\varPhi(X)+\varPsi(Y) ]<\infty$ ensures that the generalized distance covariance $V^2(X,Y)$ is finite, cf. Lemma \ref{thm:v2bound}. Define $\mu_\epsilon$ and $\nu_\epsilon$ as in \eqref{eq:munu-trunc}, and write $V^2_\epsilon(X,Y)$ for the corresponding generalized distance covariance and $\VN_\epsilon^2(\boldsymbol{X},\boldsymbol{Y})$ for its estimator. By the triangle inequality we obtain
\begin{align*}
&|V^2(X,Y)-\VN^2(\boldsymbol{X},\boldsymbol{Y})|
\\
&\qquad\leq |V^2(X,Y)-V^2_\epsilon(X,Y)| +
|V^2_\epsilon(X,Y)-\VN_\epsilon^2(\boldsymbol{X},
\boldsymbol{Y})|
\\
&\qquad\qquad\mbox{} + |\VN_\epsilon^2(\boldsymbol{X},\boldsymbol{Y})-
\VN^2(\boldsymbol{X},\boldsymbol{Y})|.
\end{align*}
    We consider the three terms on the right-hand side separately. The first term vanishes as $\epsilon\to 0$, since
\begin{equation*}
\lim_{\epsilon\to 0}V^2_\epsilon(X,Y) =
V^2(X,Y)
\end{equation*}
    by monotone convergence. For each $\epsilon>0$, the second term converges to zero as $N\to\infty$, since
\begin{equation*}
\lim_{N\to \infty} \VN_\epsilon^2(\boldsymbol{X},\boldsymbol{Y}) =
V^2_\epsilon(X,Y)\quad\text{a.s.}
\end{equation*}
    by the strong law of large numbers (SLLN) for $V$-statistics; note that this is applicable since the functions $\varPhi_\epsilon$ and $\varPsi_\epsilon$ are bounded (because of the finiteness of the L\'evy measures $\mu_\epsilon$ and $\nu_\epsilon$).

    For the third term we set $\mu^\epsilon = \mu-\mu_\epsilon$, $\nu^\epsilon = \nu-\nu_\epsilon$ and write $\varPhi^\epsilon$, $\varPsi^\epsilon$ for the corresponding continuous negative definite functions. Lemma \ref{thm:v2bound} yields the inequality
\begin{align*}
&|\VN_\epsilon^2\bigl((x_1,y_1),
\dots,(x_N,y_N)\bigr)-\VN^2
\bigl((x_1,y_1),\dots,(x_N,y_N)
\bigr)|
\\
&\qquad= \iint \bigl\llvert f_{\widehat X^{(N)},\widehat Y^{(N)}}(s,t)-f_{\widehat X^{(N)}}(s)f_{\widehat Y^{(N)}}(t)
\bigr\rrvert ^2 \mu^\epsilon(\dds) \, \nu^\epsilon(\ddt)
\\
&\qquad\leq 16\, \Ee \varPhi^\epsilon\bigl(\widehat X^{(N)}\bigr)
\cdot \Ee \varPsi^\epsilon\bigl(\widehat Y^{(N)}\bigr) = 16 \sum
_{i=1}^N \frac{1}{N}
\varPhi^\epsilon(x_i) \cdot \sum_{i=1}^N
\frac{1}{N} \varPsi^\epsilon(y_i).
\end{align*}
    From the representation \eqref{eq:PsiPhi} we know that $\varPhi^\epsilon(x)\leq\varPhi(x)$, hence also $\Ee\varPhi^\epsilon(X) \leq \Ee \varPhi(X)$ and this is finite by assumption. Therefore, we can use monotone convergence to conclude that $\lim_{\epsilon\to 0}\Ee \varPhi^\epsilon(X) = 0$. Thus, the classical SLLN applies and proves
\begin{align*}
\lim_{\epsilon\to 0}\limsup_{N\to\infty} |
\VN_\epsilon^2(\boldsymbol{X},\boldsymbol{Y})-\VN^2(\boldsymbol{X},\boldsymbol{Y})|
&\leq \lim_{\epsilon\to 0} \Ee \varPhi^\epsilon(X) \cdot \Ee
\varPsi^\epsilon(Y) = 0\quad\text{a.s.} \qedhere
\end{align*}
\end{proof}

Next we study the behaviour of the estimator under the hypothesis of
independence.
\begin{theorem} \label{thm:2dconvtoGaussian}
If $X$ and $Y$ are independent and satisfy the moment conditions
\begin{equation}
\label{est-e20} \Ee \bigl[\varPhi(X)+\varPsi(Y) \bigr]<\infty \quad \text{and}\quad
\Ee \bigl[\log^{1+\epsilon}\bigl(1+|X|^2\bigr) +
\log^{1+\epsilon}\bigl(1+|Y|^2\bigr) \bigr]<\infty
\end{equation}
for some $\epsilon>0$, then
\begin{equation}
\label{est-e22} N\cdot\VN^2(\boldsymbol{X},\boldsymbol{Y}) \xrightarrow[N\to\infty]{d}
\iint|\G(s,t)|^2\,\mu(\dds)\,\nu(\ddt) = \|\G\|_{\mu\otimes \nu}^2
\end{equation}
in distribution, where $(\G(s,t))_{(s,t)\in \real^{m+n}}$ is a complex-valued
Gaussian random field with $\Ee(\G(s,t)) = 0$ and
\begin{align}
\notag&\Cov \bigl(\G(s,t),\G\bigl(s',t'\bigr) \bigr)
\\
&\quad= \bigl(f_X\bigl(s-s'
\bigr)-f_X(s)\overline{f_X\bigl(s'\bigr)}
\bigr)\cdot \bigl(f_Y\bigl(t-t'\bigr)-f_Y(t)
\overline{f_Y\bigl(t'\bigr)} \bigr).\label{est-e24}
\end{align}
\end{theorem}
\begin{proof}
Let $X,Y$ and $\G$ be as above and let $X',X_i$ and $Y',Y_i$ ($i\in\nat$) be
independent random variables with laws $\law(X)$ and $\law(Y)$, respectively.
Define for $N\in\nat$, $s\in\real^m$, $t\in \real^n$
\begin{equation}
\label{est-e26} Z_N(s,t):= \frac{1}{N}\sum
_{k=1}^N \ee^{\ii \scalp{s}{X_k}+ \ii\scalp{t}{Y_k}} - \frac{1}{N^2}
\sum_{k,l=1}^N \ee^{\ii \scalp{s}{X_k} + \ii \scalp{t}{Y_l}}.
\end{equation}
Then
\begin{equation}
\label{est-e28} N\cdot\VN^2(\boldsymbol{X},\boldsymbol{Y}) = \|\sqrt{N}
Z_N\|^2_{\mu \otimes\nu}.
\end{equation}
The essential idea is to show that (on an appropriate space) $\sqrt{N} Z_N
\xrightarrow{d} \G$ and then to apply the continuous mapping theorem.

Since $X$ and $Y$ are independent, we have
\begin{align}
&\Ee\bigl(Z_N(s,t)\bigr)= 0,\label{eq:ExpZn2-a}
\\
&\notag\Ee\bigl(Z_N(s,t)\overline{Z_N
\bigl(s',t'\bigr)}\bigr)
\\
&\quad= \tfrac{N-1}{N^2} \bigl(f_X\bigl(s-s'
\bigr)-f_X(s)\overline{f_X\bigl(s'\bigr)}
\bigr) \bigl(f_Y\bigl(t-t'\bigr)-f_Y(t)
\overline{f_Y\bigl(t'\bigr)}\bigr),
\label{eq:ExpZn2-c}
\\
&\Ee|\sqrt{N}Z_N(s,t)|^2 = \tfrac{N-1}{N}
\bigl(1-|f_X(s)|^2\bigr) \bigl(1-|f_Y(t)|^2
\bigr). \label{eq:ExpZn2-b}
\end{align}
The first identity \eqref{eq:ExpZn2-a} is obvious and \eqref{eq:ExpZn2-b}
follows from \eqref{eq:ExpZn2-c} if we set $s=s'$ and $t=t'$. The proof of
\eqref{eq:ExpZn2-c} is deferred to Lemma~\ref{lem:ExpZn2-c} following this
proof.

The convergence $\sqrt{N} Z_N \xrightarrow{d} \G$ in $\Cskript_T:= (C(K_T),
\|.\|_{K_T})$ with $K_T = \{x\in\real^{m+n}: |x|\leq T\}$, i.e. in the space
of continuous functions on $K_T$ equipped with the supremum norm, holds if
$\Ee\log^{1+\epsilon} (\sqrt{|X|^2+|Y|^2}\vee 1 ) <\infty$, see
Cs\"{o}rg{\H{o}} \cite[Thm.~on p.~294]{Csoe1985} or Ushakov
\cite[Sec.~3.7]{Usha1999}. This log-moment condition is equivalent to the
log-moment condition \eqref{est-e20}, see Lemma \ref{ar-61}.

In fact, the result in \cite{Csoe1985} is cast in a more general setting,
proving the convergence for vectors $(X_1,\dots,X_N)$, but only
one-dimensional marginals are considered. The proof for multidimensional
marginals is very similar, so we will only give an outline:

Let $F_Z$ denote the distribution function of the random variable $Z$ and
$\FN_z$ the empirical distribution of a sample $z_1,\ldots,z_N$; if the
sample is replaced by $N$ independent copies of $Z$, we write $\FN_Z$. Using
this notation and the independence of $X$ and $Y$ yields the representation
\begin{align}
\sqrt{N} Z_N(s,t) \notag &= \sqrt{N} \Biggl(\frac{1}{N} \sum
_{k=1}^N \ee^{\ii \scalp{s}{X_k} + \ii \scalp{t}{Y_k}} -
f_{(X,Y)}(s,t) \Biggr)
\\
\notag &\qquad\mbox{} - \sqrt{N} \Biggl(\frac{1}{N} \sum
_{k=1}^N \ee^{\ii \scalp{s}{X_k}} - f_X(s)
\Biggr) \Biggl( \frac{1}{N} \sum_{l=1}^N
\ee^{\ii \scalp{t}{Y_l}} \Biggr)
\\
\notag &\qquad\mbox{} - \sqrt{N} \Biggl(\frac{1}{N} \sum
_{l=1}^N \ee^{\ii \scalp{t}{Y_l}} - f_Y(t)
\Biggr) f_X(s)
\\
 \notag&= \int \ee^{\ii (\scalp sx + \scalp ty)}\,\dd\bigl(\sqrt{N} \bigl(
\FN_{(X,Y)} (x,y) - F_{(X,Y)} (x,y)\bigr)\bigr)
\\
\notag & \qquad\mbox{}- \int \ee^{\ii \scalp sx}\,\dd\bigl(\sqrt{N} \bigl(
\FN_{X} (x) - F_{X} (x)\bigr)\bigr) \cdot \Biggl(
\frac{1}{N} \sum_{l=1}^N
\ee^{\ii \scalp{t}{Y_l}} \Biggr)
\\
 & \qquad\mbox{}- \int \ee^{\ii \scalp ty}\,\dd\bigl(\sqrt{N} \bigl(
\FN_{Y} (y) - F_{Y} (y)\bigr)\bigr) \cdot
f_X(s).\label{eq:Zn-empircalcharfun-minus-charfun}
\end{align}
Note that for bivariate distributions the integrals with respect to a single
variable reduce to integrals with respect to the corresponding marginal
distribution, e.g.\break $\int_{\real^{m+n}} h(x) dF_{(X,Y)}(x,y) = \int_{\real^m}
h(x) dF_X(x)$. Therefore, a straightforward calculation shows that $\sqrt{N}
Z_N(s,t)$ equals
\begin{equation}
\int g(x,y) \ \dd\bigl(\sqrt{N} \bigl(\FN_{(X,Y)} (x,y) -
F_{(X,Y)} (x,y)\bigr)\bigr)
\end{equation}
with the integrand
\begin{equation*}
g(x,y) := \ee^{\ii (\scalp sx + \scalp ty)} - \ee^{\ii \scalp sx} \Biggl[{\textstyle
\frac{1}{N} \sum\limits
_{l=1}^N
\ee^{\ii \scalp{t}{Y_l}}} \Biggr] - f_X(t) \ee^{\ii \scalp ty}.
\end{equation*}
Following Cs\"{o}rg{\H{o}} \cite{Csoe1981a} we obtain for $N\to \infty$ the
limit
\begin{equation}
\label{est-e32} \int_{\real^{m+n}} \bigl(\ee^{\ii \scalp sx + \ii \scalp ty} -
\ee^{\ii \scalp sx} f_Y(t) - f_X(s) \ee^{\ii \scalp ty}
\bigr)\mathrm{d}B\xch{(x,y),}{(x,y)}
\end{equation}
where $B$ is a Brownian bridge; as in~\cite[Eq.~(3.2)]{Csoe1981a} one can
show that it is a Gaussian process indexed by $\real^{m+n}$ satisfying
\begin{align}
\Ee\bigl(B(x,y)\bigr) &=\label{est-e34} \xch{0,}{0}
\\
\Ee \bigl(B(x,y)B\bigl(x',y'\bigr) \bigr)
&= \Pp \bigl(X\leq x\wedge x', Y\leq y\wedge y' \bigr)\notag
\\
&\label{est-e36}\qquad\mbox{} - \Pp (X\leq x, Y\leq y ) \Pp \xch{\bigl(X\leq
x',Y\leq y' \bigr),}{\bigl(X\leq
x',Y\leq y' \bigr)}
\\
\lim_{x\to-\infty} B(x,y) &=\label{est-e38} \lim_{y\to -\infty}
B(x,y) = \lim_{(x,y)\to (\infty,\infty)}B(x,y) = 0.
\end{align}
The limit \eqref{est-e32} is continuous if, and only if, a rather complicated
tail condition is satisfied \cite[Thm.~3.1]{Csoe1981a}; Cs\"{o}rg{\H{o}}
\cite[p.~294]{Csoe1985} shows that this condition is implied by the simpler
moment condition \eqref{est-e20}, cf. Lemma~\ref{ar-61}. Thus, $\sqrt{N} Z_N
\xrightarrow{d} \G$ in $\Cskript_T:= (C(K_T), \|.\|_{K_T})$.

Pick $\epsilon>0$, set $\delta:= 1/\epsilon$ and define
\begin{equation}
\label{est-e40} \mu_{\epsilon,\delta}(A) := \mu\bigl(A\cap \{\epsilon \leq |s| <
\delta\}\bigr) \et \mu^{\epsilon,\delta} := \mu-\mu_{\epsilon,\delta};
\end{equation}
the measures $\nu_{\epsilon,\delta}$ and $\nu^{\epsilon,\delta}$ are defined
analogously. Note that
\begin{align}
\bigl\llvert \|h\|_{\mu_{\epsilon,\delta}\otimes\nu_{\epsilon,\delta}} -
\|h'\|_{\mu_{\epsilon,\delta}\otimes\nu_{\epsilon,\delta}} \bigr\rrvert ^2 &\leq
\|h-h'\|_{\mu_{\epsilon,\delta}\otimes\nu_{\epsilon,\delta}}^2
\notag\\
&= \int |h-h'|^2\,\dd\mu_{\epsilon,\delta}\otimes
\nu_{\epsilon,\delta}
\notag\\
&\leq \|h-h'\|^2_{K_T} \cdot
\mu_{\epsilon,\delta}\bigl(\real^n\bigr) \cdot \nu_{\epsilon,\delta}\bigl(
\real^m\bigr)
\end{align}
shows that $h\mapsto
\|h\|^2_{\mu_{\epsilon,\delta}\otimes\nu_{\epsilon,\delta}}$ is continuous on
$\Cskript_T$. Thus, the continuous mapping theorem implies
\begin{equation}
\label{est-e44} \|\sqrt{N} Z_N\|^2_{\mu_{\epsilon,\delta}\otimes\nu_{\epsilon,\delta}}
\xrightarrow[N\to\infty]{d} \|\G\|^2_{\mu_{\epsilon,\delta}\otimes\nu_{\epsilon,\delta}}.
\end{equation}
By the triangle inequality we have
\begin{align}
\notag \bigl\llvert N\cdot\VN^2(\boldsymbol{X},\boldsymbol{Y}) - \|\G
\|_{\mu\otimes \nu}^2 \bigr\rrvert &\leq \bigl\llvert N\cdot
\VN^2(\boldsymbol{X},\boldsymbol{Y}) - \|\sqrt{N} Z_N\|^2_{\mu_{\epsilon,\delta}\otimes\nu_{\epsilon,\delta}}
\bigr\rrvert
\\
&\notag\qquad\mbox{} + \bigl\llvert \|\sqrt{N} Z_N\|^2_{\mu_{\epsilon,\delta}\otimes\nu_{\epsilon,\delta}}-
\|\G\|^2_{\mu_{\epsilon,\delta}\otimes\nu_{\epsilon,\delta}} \bigr\rrvert
\\
&\qquad\mbox{} + \bigl\llvert \|\G\|^2_{\mu_{\epsilon,\delta}\otimes\nu_{\epsilon,\delta}}- \|\G
\|_{\mu\otimes \nu}^2 \bigr\rrvert .
\end{align}
Thus, it remains to show that the first and last terms on the right-hand side
vanish uniformly as $\epsilon\to 0$. Note that
\begin{align}
\notag&\|\G\|^2_{\mu\otimes\nu} - \|\G
\|^2_{\mu_{\epsilon,\delta}\otimes\nu_{\epsilon,\delta}}
\\
&\quad =\|\G\|^2_{\mu^{\epsilon,\delta}\otimes \nu^{\epsilon,\delta}} + \|\G\|^2_{ \mu_{\epsilon,\delta}\otimes \nu^{\epsilon,\delta}}
+ \|\G\|^2_{ \mu^{\epsilon,\delta}\otimes \nu_{\epsilon,\delta}} \xrightarrow[\epsilon \to 0]{} 0 \
\text{a.s.} \label{eq:gauss-on-bounded}
\end{align}
This follows from the dominated convergence theorem, since
\begin{align}
\notag\Ee \bigl(\|\G\|^2_{\mu\otimes\nu} \bigr) &= \bigl\llVert 1-|f_X|^2 \bigr\rrVert
^2_{\mu} \cdot \bigl\llVert 1-|f_Y|^2
\bigr\rrVert ^2_{\nu}
\\
&= \Ee\varPhi\bigl(X-X'\bigr) \cdot \Ee\varPsi
\bigl(Y-Y'\bigr) < \infty. \label{eq:EG}
\end{align}
Moreover,
\begin{align}
\Ee & \bigl( \bigl\llvert N\cdot\VN^2(
\boldsymbol X, \boldsymbol Y) - \|\sqrt{N} Z_N\|^2_{\mu_{\epsilon,\delta}\otimes\nu_{\epsilon,\delta}} \bigr
\rrvert \bigr)
\notag\\
&= \Ee \|\sqrt{N} Z_N\|^2_{ \mu^{\epsilon,\delta}\otimes \nu^{\epsilon,\delta}} + \Ee \|
\sqrt{N} Z_N\|^2_{ \mu_{\epsilon,\delta}\otimes \nu^{\epsilon,\delta}} + \Ee \|\sqrt{N}
Z_N\|^2_{ \mu^{\epsilon,\delta}\otimes \nu_{\epsilon,\delta}} \label{est-e50}
\end{align}
and for the first term we have
\begin{equation}
\label{eq:Zn-on-bounded} \bigl\llVert \Ee\bigl(|\sqrt{N} Z_N|^2
\bigr) \bigr\rrVert ^2_{ \mu^{\epsilon,\delta}\otimes \nu^{\epsilon,\delta}} = \bigl(\tfrac{N-1}{N}
\bigr)^2\cdot \bigl\llVert 1-|f_X|^2 \bigr
\rrVert ^2_{ \mu^{\epsilon,\delta}} \cdot \bigl\llVert 1-|f_Y|^2
\bigr\rrVert ^2_{\nu^{\epsilon,\delta}} \xrightarrow[\epsilon\to 0]{} 0
\end{equation}
by dominated convergence, since we have
$\|1-|f_X|^2\|^2_{\mu^{\epsilon,\delta}} \leq \Ee\varPhi(X-X') < \infty$ and
$\|1-|f_Y|^2\|^2_{ \nu^{\epsilon,\delta}} \leq \Ee\varPsi(Y-Y') < \infty$.
The other summands are dealt with similarly.

The result follows since the convergence in \eqref{eq:gauss-on-bounded} and
\eqref{eq:Zn-on-bounded} is uniform in $N$.
\end{proof}

We still have to prove \eqref{eq:ExpZn2-c}.
\begin{lemma}\label{lem:ExpZn2-c}
    In the setting of \textup{(}the proof of\textup{)} Theorem~\ref{thm:2dconvtoGaussian} we have
\begin{align*}
&\Ee\bigl(Z_N(s,t)\overline{Z_N\bigl(s',t'
\bigr)}\bigr)
\\
&\quad= \tfrac{N-1}{N^2} \bigl(f_X\bigl(s-s'
\bigr)-f_X(s)\overline{f_X\bigl(s'\bigr)}
\bigr) \bigl(f_Y\bigl(t-t'\bigr)-f_Y(t)
\overline{f_Y\bigl(t'\bigr)}\bigr).
\end{align*}
\end{lemma}
\begin{proof}
    Observe that
\begin{align*}
Z_N(s,t) &= \frac{1}{N}\sum_{k=1}^N
\ee^{\ii \scalp{s}{X_k}+ \ii\scalp{t}{Y_k}} - \frac{1}{N^2} \sum_{k,l=1}^N
\ee^{\ii \scalp{s}{X_k} + \ii \scalp{t}{Y_l}}
\\
&= \frac{1}{N}\sum_{k=1}^N
\Biggl(\ee^{\ii \scalp{s}{X_k}} - {\textstyle\frac{1}{N} \sum
\limits_{l=1}^N
\ee^{\ii \scalp{s}{X_l}}} \Biggr) \Biggl(\ee^{\ii \scalp{t}{Y_k}} - {\textstyle
\frac{1}{N} \sum\limits
_{l=1}^N
\ee^{\ii \scalp{t}{Y_l}}} \Biggr).
\end{align*}
    Using this formula and the independence of the random variables $(X_1,\dots,X_N)$ and $(Y_1,\dots,Y_N)$ yields
\begin{align*}
&\Ee\bigl(Z_N(s,t)\overline{Z_N\bigl(s',t'
\bigr)}\bigr)
\\
&= \frac{1}{N^2} \sum_{j,k=1}^N \Ee
\Biggl[ \Biggl(\ee^{\ii \scalp{s}{X_k}} - {\textstyle\frac{1}{N} \sum
\limits
_{l=1}^N \ee^{\ii \scalp{s}{X_l}}} \Biggr) \Biggl(
\ee^{\ii \scalp{t}{Y_k}} - {\textstyle\frac{1}{N} \sum\limits
_{l=1}^N
\ee^{\ii \scalp{t}{Y_l}}} \Biggr)
\\
&\qquad\qquad\qquad\qquad\mbox{}\times \Biggl(\ee^{-\ii \scalp{s'}{X_j}} - {\textstyle
\frac{1}{N} \sum\limits
_{l=1}^N
\ee^{-\ii \scalp{s'}{X_l}}} \Biggr) \Biggl(\ee^{-\ii \scalp{t'}{Y_j}} - {\textstyle
\frac{1}{N} \sum\limits
_{l=1}^N
\ee^{-\ii \scalp{t'}{Y_l}}} \Biggr) \Biggr]
\\
&= \frac{1}{N^2} \sum_{j,k=1}^N \Ee
\Biggl[ \Biggl(\ee^{\ii \scalp{s}{X_k}} - {\textstyle\frac{1}{N} \sum
\limits
_{l=1}^N \ee^{\ii \scalp{s}{X_l}}} \Biggr) \Biggl(
\ee^{-\ii \scalp{s'}{X_j}} - {\textstyle\frac{1}{N} \sum\limits
_{l=1}^N
\ee^{-\ii\scalp{s'}{X_l}}} \Biggr) \Biggr]
\\
&\qquad\qquad\qquad\qquad\mbox{}\times \Ee \Biggl[ \Biggl(\ee^{\ii \scalp{t}{Y_k}} - {
\textstyle\frac{1}{N} \sum\limits
_{l=1}^N
\ee^{\ii \scalp{t}{Y_l}}} \Biggr) \Biggl(\ee^{-\ii \scalp{t'}{Y_j}} - {\textstyle
\frac{1}{N} \sum\limits
_{l=1}^N
\ee^{-\ii \scalp{t'}{Y_l}}} \Biggr) \Biggr].
\end{align*}
    A lengthy but otherwise straightforward calculation shows that
\begin{align*}
&\Ee \Biggl[ \Biggl(\ee^{\ii \scalp{s}{X_k}} - {\textstyle\frac{1}{N} \sum
\limits
_{l=1}^N \ee^{\ii \scalp{s}{X_l}}} \Biggr) \Biggl(
\ee^{-\ii \scalp{s'}{X_j}} - {\textstyle\frac{1}{N} \sum\limits
_{l=1}^N
\ee^{-\ii \scalp{s'}{X_l}}} \Biggr) \Biggr]
\\
&\qquad= \Ee \bigl(\ee^{\ii \scalp{s}{X_k} - \ii \scalp{s'}{X_j}} \bigr) - \frac{N-1}{N}
f_X(s)\overline{f_X\bigl(s'\bigr)} -
\frac{1}N f_X\bigl(s-s'\bigr),
\end{align*}
    and an analogous formula holds for the $Y_i$. Summing over $k,j=1,\dots,N$ and distinguishing between the cases $k=j$ and $k\neq j$ finally gives
\begin{align*}
&\Ee\bigl(Z_N(s,t)\overline{Z_N\bigl(s',t'
\bigr)}\bigr)
\\
&= \biggl(\frac{(N-1)^2}{N^3} + \frac{N-1}{N^3} \biggr) \bigl(f_X
\bigl(s-s'\bigr)-f_X(s)\overline{f_{X}
\bigl(s'\bigr)}\bigr) \xch{\bigl(f_Y\bigl(t-t'
\bigr)-f_Y(t)\overline{f_{Y}\bigl(t'\bigr)}
\bigr),}{\bigl(f_Y\bigl(t-t'
\bigr)-f_Y(t)\overline{f_{Y}\bigl(t'\bigr)}
\bigr)}
\end{align*}
    and the lemma follows.
\end{proof}

\begin{remark}\label{rem:momcond}
a)  If we symmetrize the expression for $N\cdot\VN^2$ in a suitable way, we can transform it into a degenerate U-statistic. For random variables with bounded second moments we can then use classical results to show the convergence to $\sum_{k=1}^\infty \lambda_k X_k^2$, where $\lambda_k$ are some coefficients and $X_k$ are i.i.d.  standard normal random variables, see e.g. Serfling~\cite[Sec.~5.5.2]{Serf2009} or Witting \& M\"{u}ller-Funk~\cite[Satz~7.183]{WittMuel1995}. In order to relax the bounds on the moments one would have to show convergence of the corresponding $\lambda_k$.

b)      The log-moment condition can be slightly relaxed, but this leads to a much more involved expression, cf. Cs\"{o}rg{\H{o}}~\cite{Csoe1985}, and for the case $\epsilon = 0$ a counterexample is known, see Cs\"{o}rg{\H{o}}~\cite[p.~133]{Csor1981}. Unfortunately, the convergence of the characteristic functions is stated without any moment condition in Murata~\cite[Thm.~4]{Mura2001} which is based on Feuerverger \& Mureika~\cite[Thm.~3.1]{FeueMure1977}.
\end{remark}

\begin{corollary}\label{est-25}
    Assume that $X$, $Y$ are non-degenerate with $\Ee \varPhi(X)+\Ee \varPsi(Y)<\infty$ and set
\begin{gather*}
a_N:= \frac{1}{N^2} \sum_{i,k = 1}^N
\varPhi(X_i-X_k) \et b_N:=
\frac{1}{N^2} \sum_{j,l = 1}^N
\varPsi(Y_j-Y_l).
\end{gather*}
\begin{enumerate}
    \item\label{est-25-a} If $X$, $Y$ are independent random variables
        satisfying the log-moment conditions $\Ee
        \log^{1+\epsilon}(1+|X|^2) <\infty$ and $\Ee
        \log^{1+\epsilon}(1+|Y|^2) <\infty$ for some $\epsilon>0$, then
\begin{equation}
\label{est-e60} \frac{N\cdot\VN^2}{a_Nb_N} \xrightarrow[N\to\infty]{d} \sum
_{k=1}^\infty \alpha_k
X_k^2,
\end{equation}
     where $X_i$ are independent standard normal random variables and the
     coefficients $\alpha_k$ satisfy $\alpha_k \geq 0$ with
     $\sum_{k=1}^\infty \alpha_k = 1$.

    \item\label{est-25-b} If the random variables $X$ and $Y$ are not
        independent, then
\begin{equation}
\frac{N\cdot\VN^2}{a_Nb_N} \xrightarrow[N\to\infty]{d} \infty.
\end{equation}
    \end{enumerate}
\end{corollary}
\begin{proof}
    We can almost literally follow the proof of Corollary 2 and Theorem 6 in Sz\'ekely et al.~\cite{SzekRizzBaki2007}. Just note that $a_Nb_N$ is an estimator for $\Ee\varPhi(X-X') \cdot \Ee\varPsi(Y-Y')$ and this is, by \eqref{eq:ExpZn2-b} and \eqref{eq:EG}, the limit of the expectation of the numerator.
\end{proof}

\section{Remarks on the uniqueness of the Cauchy distance covariance}\label{uni}
It is a natural question whether it is possible to extend distance covariance
further by taking a measure $\rho$ in the definition \eqref{eq:v2-def} which
does not factorize. To ensure the finiteness of $V(X,Y)$ one still has to
assume
\begin{equation}
\label{uni-e02} \int \bigl(1\wedge|s|^2\bigr) \bigl(1
\wedge|t|^2\bigr) \,\rho(\dds,\ddt) <\infty,
\end{equation}
see also Sz\'ekely \& Rizzo \cite[Eq.~(2.4)]{SzekRizz2012} and
\eqref{eq:v2-def} at the beginning of Section~\ref{sec:gdc-infinite}.
Furthermore, it is no restriction to assume that $\rho$ is symmetric, since
the integrand in \eqref{eq:v2-def} is symmetric, hence we can always
symmetrize any non-symmetric measure $\rho$ without changing the value of the
integral. Thus, the function
\begin{equation}
\label{uni-e04} \varTheta(x,y) := \int_{\real^n}\int
_{\real^m} \bigl(1-\cos(\scalp xs)\bigr) \bigl(1-\cos(\scalp xt)
\bigr) \, \rho(\dds,\ddt)
\end{equation}
is well-defined and symmetric in each variable. The corresponding generalized
distance covariance $V^2(X,Y)$ can be expressed by \eqref{eq:v2-ThetaXY}, if
the expectations on the right-hand side are finite. Note, however, that the
nice and computationally feasible representations of $V^2(X,Y)$ make
essential use of the factorization of $\rho$ which means that they are no
longer available in this setting.

Let $X$ and $Y$ be random variables with values in $\real^m$ and $\real^n$,
respectively, such that for some $x\in\real^m$ and $y\in\real^n$
\begin{equation}
\label{uni-e06} \Pp(X=0) = \Pp(X=x) = \Pp(Y=0) = \Pp(Y=y) = \frac{1}{2}.
\end{equation}
A direct calculation of \eqref{eq:v2-ThetaXY}, using $\varTheta(0,y) =
\varTheta(x,0) = 0$, gives
\begin{equation}
\label{uni-e08} V^2(X,Y) = \gamma\cdot \varTheta\xch{(x,y),}{(x,y)}
\end{equation}
with
\begin{align}
\gamma&{} = 2\Pp(X_1\neq
X_5,Y_1\neq Y_6) - \Pp(X_1
\neq X_4,Y_1\neq Y_4)\notag\\
&\quad{} - \Pp(X_2\neq X_5)\Pp\xch{(Y_3\neq Y_6),}{(Y_3\neq Y_6)}
\label{uni-e10}
\end{align}
where $(X_i,Y_i), i=1,\dots,6$, are i.i.d.  copies of the random vector
$(X,Y)$.

Now suppose that $V^2(X,Y)$ is homogeneous and/or rotationally invariant,
i.e. for some $\alpha,\beta\in (0,2)$, all scalars $a,b>0$ and  orthogonal
matrices $A\in\real^{n\times n}, B\in\real^{m\times m}$
\begin{align}
V^2(aX,bY) &=\label{eq:Vhomo} a^\alpha b^\beta
V^2\xch{(X,Y),}{(X,Y)}
\\
V^2(AX,BY) &=\label{eq:Vrotin} V^2(X,Y)
\end{align}
hold. The homogeneity, \eqref{eq:Vhomo}, yields
\begin{equation}
\varTheta(x,y) = |x|^\alpha|y|^\beta \varTheta \xch{\bigl(
\tfrac{x}{|x|},\tfrac{y}{|y|} \bigr),}{\bigl(
\tfrac{x}{|x|},\tfrac{y}{|y|} \bigr)}
\end{equation}
and the rotational invariance, \eqref{eq:Vrotin}, shows that
$\varTheta(x/|x|, y/|y|)$ is a constant. In particular, homogeneity of degree
$\alpha=\beta=1$ and rotational invariance yield that  $\varTheta(x,y) =
\text{const}\cdot|x|\cdot |y|$. Since the L\'evy--Khintchine formula
furnishes a one-to-one correspondence between the cndf and its L\'evy
triplet, see (the comments following) Theorem~\ref{ar-03}, this determines
$\rho$ uniquely: it factorizes into two Cauchy L\'evy measures. This means
that~-- even in a larger class of weights~-- the assumptions \eqref{eq:Vhomo}
and \eqref{eq:Vrotin} imply a unique (up to a constant) choice of weights,
and we have recovered Sz\'ekely-and-Rizzo's uniqueness result from
\cite{SzekRizz2012}.
\begin{lemma}\label{uni-03}
    Let $V^2(X,Y) := \|f_{X,Y}-f_X\otimes f_Y\|^2_{L^2(\rho)}$ be a generalized distance covariance as in Definition~\ref{ar-05} and assume that the symmetric measure $\rho$ satisfies the integrability condition \eqref{uni-e02}. If $V^2(X,Y)$ is homogeneous of order $\alpha\in (0,2)$ and $\beta\in (0,2)$ and rotationally invariant in each argument, then the measure $\rho$ defining $V^2(X,Y)$ is~-- up to a multiplicative constant~-- of the form
\begin{gather*}
\rho(\dds,\ddt) = c(\alpha,m)c(\beta,n)|s|^{-\alpha-m}|t|^{-\beta-n}\,
\dds\,\ddt.
\end{gather*}
    Moreover, $V^2(X,Y)$ can be represented by \eqref{eq:v2-ThetaXY} with $\varTheta(x,y)=C\cdot|x|^\alpha\cdot|y|^\beta$.
\end{lemma}
For completeness, let us mention that the constants $c(\alpha,m)$ and
$c(\beta,n)$ are of the form
\begin{gather*}
c(\alpha,m) = \alpha 2^{\alpha-1}\pi^{-m/2}\varGamma \bigl(
\tfrac{\alpha+m}{2} \bigr) /\varGamma \bigl(1-\tfrac{\alpha}2 \bigr),
\end{gather*}
see e.g. \cite[p.~34, Example~2.4.d)]{BoetSchiWang2013} or \cite[p.~184,
Exercise~18.23]{BergFors75}.

\section{Generalized distance correlation}\label{corr}

We continue our discussion in the setting of Section~\ref{sec:gdc-infinite}.
Let $\rho = \mu\otimes\nu$ and assume that $\mu$ and $\nu$ are symmetric
L\'evy measures on $\rmohne$ and $\rnohne$, each with full support. For $m$-
and $n$-dimensional random variables $X$ and $Y$ the generalized distance
covariance is, cf.~Definition \ref{gdc-05},
\begin{equation}
\label{corr-e02} V(X,Y) = \sqrt{\iint |f_{(X,Y)}(s,t)-f_X(s)f_Y(t)|^2
\,\mu(\dds)\,\nu(\ddt)}.
\end{equation}
We set
\begin{align}
    V(X):=&\notag \sqrt{\int |f_{(X,X)}(s,t)-f_X(s)f_X(t)|^2 \,\mu(\dds)\,\mu(\ddt)}\\
    =&\label{corr-e04} \|f_{(X,X)}-f_X\otimes f_X\|_{L^2(\mu \otimes \mu)},\\
    V(Y):=&\notag \sqrt{\int |f_{(Y,Y)}(s,t)-f_Y(s)f_Y(t)|^2 \,\nu(\dds)\,\nu(\ddt)}\\
    =&\label{corr-e06} \|f_{(Y,Y)}-f_Y\otimes f_Y\|_{L^2(\nu \otimes \nu)},
\end{align}
and define \emph{generalized distance correlation} as
\begin{equation}
\label{corr-e10} R(X,Y):= \begin{cases}
        \dfrac{V(X,Y)}{\sqrt{V(X) V(Y)}}, & \text{if\ \ } V(X)\cdot V(Y)>0,\\
        0, & \text{otherwise.}
    \end{cases}
\end{equation}
Using the Cauchy--Schwarz inequality it follows from \eqref{gdc-e24} that,
whenever $R(X,Y)$ is well defined, one has
\begin{equation}
\label{corr-e12} 0\leq R(X,Y) \leq 1 \et R(X,Y) = 0\text{\ iff\ } X, Y \text{\ are
independent.}
\end{equation}

The \emph{sample distance correlation} is given by
\begin{equation}
\label{corr-e20} \RN\bigl((x_1,y_1),\dots,(x_N,y_N)
\bigr) = \Biggl(\frac%
 {\frac{1}{N^2} \sum
_{k,l=1}^N A_{kl}B_{kl}}
 {\sqrt{{ \frac{1}{N^2} \sum_{k,l=1}^N
A_{kl}A_{kl}}} \cdot \sqrt{{\frac{1}{N^2}
\sum_{k,l=1}^N B_{kl}B_{kl}}}}
 \Biggr)^{\frac{1} {2}},
\end{equation}
where we use the notation introduced in Lemma~\ref{est-05}.

\begin{example}\label{corr-11}
    For standard normal random variables $X, Y$ with $\rho = \Cor(X,Y)$, $\varPhi(x) = |x|$ and $\varPsi(y) = |y|$ the distance correlation becomes
\begin{equation}
\label{corr-e22} R(X,Y) = \biggl(\frac{\sqrt{1-\rho^2}-\sqrt{4-\rho^2} + \rho(\arcsin \rho - \arcsin \frac{\rho}{2}) + 1}{1-\sqrt{3}+\pi/3} \biggr)^{1/2} \leq |
\rho|,
\end{equation}
    cf.~Sz\'ekely \& Rizzo \cite[Thm.~6]{SzekRizz2009}.
\end{example}

\section{Gaussian covariance}\label{gauss}
Let us finally show that the results on Brownian covariance of Sz\'ekely \&
Rizzo \cite[Sec.~3]{SzekRizz2009} do have an analogue for the generalized
distance covariance.

For a symmetric L\'evy measure with corresponding continuous negative
definite function $\varPhi :\real^m \to \real$ let $G_\varPhi$ be the
Gaussian field indexed by $\real^m$ with
\begingroup
\abovedisplayskip=6.5pt
\belowdisplayskip=6.5pt
\begin{equation}
\label{gauss-e02} \Ee G_\varPhi(x) = 0 \et \Ee\bigl(G_\varPhi(x)G_\varPhi
\bigl(x'\bigr)\bigr) = \varPhi(x) + \varPhi\bigl(x'
\bigr) - \varPhi\bigl(x-x'\bigr).
\end{equation}
Analogously we define the random field $G_\varPsi$.

For a random variable $Z$ with values in $\real^d$ and for a Gaussian random
field $G$ indexed by $\real^d$ we set
\begin{equation}
\label{gauss-e04} Z^G := G(Z) - \Ee\bigl(G(Z) \mid G\bigr).
\end{equation}
\begin{definition}
    Let $G_\varPhi$, $G_\varPsi$ be mean-zero Gaussian random fields indexed by $\real^m$ and $\real^n$ and with covariance structure given by the cndfs $\varPhi$ and $\varPsi$, respectively. For any two $m$- and $n$-dimensional random variables $X$ and $Y$ the \emph{Gaussian covariance} is defined as
\begin{equation}
\label{gauss-e06} G^2(X,Y) := \Cov_{G_\varPhi,G_\varPsi}^2(X,Y)
:= \Ee\bigl(X_1^{G_\varPhi} X_2^{G_\varPhi}
Y_1^{G_\varPsi} Y_2^{G_\varPsi}\bigr),
\end{equation}
    where $(X_1,Y_1), (X_2,Y_2)$ are i.i.d.  copies of $(X,Y)$.
\end{definition}

We can now identify Gaussian covariance and generalized distance covariance.
\begin{theorem}[Gaussian covariance is generalized distance covariance]
Assume that $X$ and $Y$ satisfy $\Ee\varPhi(X)+\Ee\varPsi(Y)<\infty$. If
$G_\varPhi$ and $G_\varPsi$ are independent, then
\begin{equation}
\label{gauss-e10} G^2(X,Y) = V^2(X,Y).
\end{equation}
\end{theorem}
\endgroup

\begin{proof}
    The proof is similar to Sz\'ekely \& Rizzo \cite[Thm.~8]{SzekRizz2009}. By conditioning and the independence of $G_\varPhi$ and $G_\varPsi$ we see
\begin{align} &\Ee \bigl(X_1^{G_\varPhi}
X_2^{G_\varPhi} Y_1^{G_\varPsi}
Y_2^{G_\varPsi} \bigr)
\notag\\
&= \Ee \bigl(\Ee \bigl(X_1^{G_\varPhi} X_2^{G_\varPhi}
\mid X_1,X_2,Y_1,Y_2 \bigr) \cdot
\Ee \bigl(Y_1^{G_\varPsi} Y_2^{G_\varPsi} \mid
X_1,X_2,Y_1,Y_2 \bigr) \bigr).\label{eq:Wcondtioning}
\end{align}
    Using $\Ee(G_\varPhi(x)G_\varPhi(x')) = \varPhi(x) + \varPhi(x') - \varPhi(x-x') =: \varphi(x,x')$ yields
\begin{align*}
\notag &\Ee \bigl(X_1^{G_\varPhi} X_2^{G_\varPhi}
\mid X_1,X_2,Y_1,Y_2 \bigr)
\\
&= \varphi(X_1,X_2) - \Ee\bigl(\varphi(X_1,X_2)
\mid X_1\bigr)- \Ee\bigl(\varphi(X_1,X_2)
\mid X_2\bigr) + \Ee\varphi(X_1,X_2)
\\
& = - \varPhi(X_1-X_2)+\Ee\bigl(\varPhi(X_1-X_2)
\mid X_1\bigr) + \Ee\bigl(\varPhi(X_1-X_2)
\mid X_2\bigr)
\\
&\qquad\mbox{} - \Ee\varPhi(X_1-X_2),
\end{align*}
    where the second equality is due to cancellations. An analogous calculation for\break $\Ee(Y_1^{G_\varPsi} Y_2^{G_\varPsi} \mid X_1,X_2,Y_1,Y_2)$ turns \eqref{eq:Wcondtioning} into \eqref{eq:Vdoublecentered}.

    For \eqref{eq:Wcondtioning} we have to make sure that  $\Ee (|X_1^{G_\varPhi} X_2^{G_\varPhi} Y_1^{G_\varPsi} Y_2^{G_\varPsi}| ) <\infty$. This follows from
    \begin{align*}
        &\Ee\left(|X_1^{G_\Phi} X_2^{G_\Phi} Y_1^{G_\Psi} Y_2^{G_\Psi}|\right)\\
        &= \Ee\left(
            \Ee\left[|X_1^{G_\Phi} X_2^{G_\Phi}| \;\Big|\;  X_1,X_2,Y_1,Y_2\right]
            \Ee\left[|Y_1^{G_\Psi} Y_2^{G_\Psi}| \;\Big|\;  X_1,X_2,Y_1,Y_2\right]
            \right)\\
        &\leq \Ee\bigg(
            \sqrt{\Ee\left[|X_1^{G_\Phi}|^2 \;\Big|\; X_1,X_2,Y_1,Y_2\right]
            \Ee\left[|X_2^{G_\Phi}|^2 \;\Big|\; X_1,X_2,Y_1,Y_2\right]}\\
        &\qquad\quad\mbox{}\times
            \sqrt{\Ee\left[|Y_1^{G_\Psi}|^2 \;\Big|\; X_1,X_2,Y_1,Y_2\right]
            \Ee\left[|Y_2^{G_\Psi}|^2 \;\Big|\; X_1,X_2,Y_1,Y_2\right]}\bigg)\\
        &= \Ee\bigg(
            \sqrt{\Ee\left[|X_1^{G_\Phi}|^2 \;\Big|\; X_1,Y_1\right]
            \Ee\left[|X_2^{G_\Phi}|^2 \;\Big|\; X_2,Y_2\right]}\\
        &\qquad\quad\mbox{}\times
            \sqrt{\Ee\left[|Y_1^{G_\Psi}|^2 \;\Big|\; X_1,Y_1\right]
            \Ee\left[|Y_2^{G_\Psi}|^2 \;\Big|\; X_2,Y_2\right]}\bigg)\\
        &= \Ee\bigg(
            \sqrt{\Ee\left[|X_1^{G_\Phi}|^2 \;\Big|\; X_1,Y_1\right]}\bigg)^2
            \Ee\bigg(\sqrt{
            \Ee\left[|Y_1^{G_\Psi}|^2 \;\Big|\; X_1,Y_1\right]}\bigg)^2\\
        &\leq \Ee\left(|X_1^{G_\Phi}|^2\right) \Ee\left(|Y_1^{G_\Psi}|^2\right).
    \end{align*}
    In this calculation we use first the independence of $G_\Phi$ and $G_\Psi$, the
    conditional Cauchy--Schwarz inequality and the fact that the random variables $(X_1,Y_1)$ and\break $(X_2,Y_2)$ are i.i.d. In the final estimate we use again the Cauchy--Schwarz inequality.
    In order to see that the right-hand side is finite, we note that \eqref{gauss-e02} and \eqref{gauss-e04} yield
    \begin{align*}
        \Ee\left(|X^{G_\Phi}|^2\right)
        &= \Ee\left(|G_\Phi(X) - \Ee(G_\Phi(X) \mid G_\Phi)|^2\right)\\
        &\leq 2\Ee\left(|G_\Phi(X)|^2\right)
        = 4  \Ee\Phi(X) <  \infty.
    \end{align*}
    A similar estimate for $Y$ completes the proof.
\end{proof}

\section{Conclusion}

We have shown that the concept of distance covariance introduced by Sz\'ekely
et al. \cite{SzekRizzBaki2007} can be embedded into a more general framework
based on L\'evy measures, cf. Section \ref{sec:gdc}. In this generalized
setting\querymark{Q4} the key results for statistical applications are: the convergence of the
estimators and the fact that also the limit distribution of the (scaled) estimators is known, cf. Section \ref{sec:estimation}. Moreover~-- for applications this
is of major importance~-- the estimators have the numerically efficient
representation \eqref{est-e06}.

The results allow the use of generalized distance covariance in the tests for
independence developed for distance covariance, e.g. tests based on a general
Gaussian quadratic form estimate or resampling tests. The test statistic is
the function $T:= \frac{N\cdot\VN^2}{a_Nb_N}$ discussed in Corollary
\ref{est-25}. Using the quadratic form estimate (see \cite{SzekRizzBaki2007}
for details) its p-value can be estimated by $1-F(T)$ where $F$ is the
distribution function of the Chi-squared distribution with 1 degree of
freedom. This test and resampling tests are studied in detail in \cite{part2}
and \cite{Boet2017}, respectively. In addition, these papers contain
illustrating examples which show that the new flexibility provided by the
choice of the L\'evy measures (equivalently: by the choice of the continuous
negative definite function) can be used to improve the power of these tests.
Moreover, new test procedures using distance covariance and its
generalizations are developed in \cite{BersBoet2018}.

Finally, the presented results are also the foundation for a new approach to
testing and measuring
multivariate dependence, i.e. the mutual (in)dependence of
an arbitrary number of random vectors. This approach is developed in
\cite{part2} accompanied by extensive examples and further applications in
\cite{Boet2017}.  All functions required for the use of generalized distance
covariance in applications are implemented in the R package
\texttt{multivariance} \cite{Boett2017R-1.0.5}.

\begin{acknowledgement}
We are grateful to Ulrich Brehm (TU Dresden) for insightful discussions on
(elementary) symmetric polynomials and to Georg Berschneider (TU Dresden) who
read and commented on the whole text. We would also like to thank Gabor J.
Sz\'ekely (NSF) for advice on the current literature.
\end{acknowledgement}

\begin{funding}
Financial support for Martin Keller-Ressel by the
\gsponsor[id=GS1,sponsor-id=501100001659]{German Research Foundation}\break (DFG) under grant
\gnumber[refid=GS1]{ZUK~64} and \gnumber[refid=GS1]{KE~1736/1-1} is gratefully acknowledged.
\end{funding}



\begin{thebibliography}{40}

\bibitem{BakiSzek2011}
\begin{barticle}
\bauthor{\bsnm{Bakirov}, \binits{N.K.}},
\bauthor{\bsnm{Sz{\'e}kely}, \binits{G.J.}}:
\batitle{Brownian covariance and central limit theorem for stationary sequences}.
\bjtitle{Theory of Probability and its Applications}
\bvolume{55}(\bissue{3}),
\bfpage{371}--\blpage{394}
(\byear{2011})
\bid{doi={10.1137/S0040585X97984954}, mr={2768533}}
\end{barticle}
%
\OrigBibText
\begin{barticle}
\bauthor{\bsnm{Bakirov}, \binits{N.K.}},
\bauthor{\bsnm{Sz{\'e}kely}, \binits{G.J.}}:
\batitle{Brownian covariance and central limit theorem for stationary sequences}.
\bjtitle{Theory of Probability and its Applications}
\bvolume{55}(\bissue{3}),
\bfpage{371}--\blpage{394}
(\byear{2011})
\end{barticle}
\endOrigBibText
\bptok{structpyb}%
\endbibitem

\bibitem{BenLin}
\begin{bbook}
\bauthor{\bsnm{Banyamini}, \binits{Y.}},
\bauthor{\bsnm{Lindenstrauss}, \binits{J.}}:
\bbtitle{{G}eometric {N}onlinear {F}unctional {A}nalysis}.
\bpublisher{American Mathematical Society},
\blocation{Providence (RI)}
(\byear{2000})
\bid{mr={1727673}}
\end{bbook}
%
\OrigBibText
\begin{bbook}
\bauthor{\bsnm{Banyamini}, \binits{Y.}},
\bauthor{\bsnm{Lindenstrauss}, \binits{J.}}:
\bbtitle{{G}eometric {N}onlinear {F}unctional {A}nalysis}.
\bpublisher{American Mathematical Society},
\blocation{Providence (RI)}
(\byear{2000})
\end{bbook}
\endOrigBibText
\bptok{structpyb}%
\endbibitem

\bibitem{BergFors75}
\begin{bbook}
\bauthor{\bsnm{Berg}, \binits{C.}},
\bauthor{\bsnm{Forst}, \binits{G.}}:
\bbtitle{{P}otential {T}heory on {L}ocally {C}ompact {A}belian {G}roups}.
\bpublisher{Springer},
\blocation{Berlin}
(\byear{1975})
\bid{mr={0481057}}
\end{bbook}
%
\OrigBibText
\begin{bbook}
\bauthor{\bsnm{Berg}, \binits{C.}},
\bauthor{\bsnm{Forst}, \binits{G.}}:
\bbtitle{{P}otential {T}heory on {L}ocally {C}ompact {A}belian {G}roups}.
\bpublisher{Springer},
\blocation{Berlin}
(\byear{1975})
\end{bbook}
\endOrigBibText
\bptok{structpyb}%
\endbibitem

\bibitem{BerrSamw2017}
\begin{botherref}
\oauthor{\bsnm{Berrett}, \binits{T.B.}},
\oauthor{\bsnm{Samworth}, \binits{R.J.}}:
Nonparametric independence testing via mutual information.
arXiv: \arxivhref{1711.06642v1}
(2017)
\end{botherref}
%
\OrigBibText
\begin{botherref}
\oauthor{\bsnm{Berrett}, \binits{T.B.}},
\oauthor{\bsnm{Samworth}, \binits{R.J.}}:
Nonparametric independence testing via mutual information.
arXiv: 1711.06642v1
(2017)
\end{botherref}
\endOrigBibText
\bptok{structpyb}%
\endbibitem

\bibitem{BersBoet2018}
\begin{botherref}
\oauthor{\bsnm{Berschneider}, \binits{G.}},
\oauthor{\bsnm{B{\"o}ttcher}, \binits{B.}}:
On complex Gaussian random fields, Gaussian quadratic forms and sample distance
 multivariance.
arXiv: \arxivhref{1808.07280v1}
(2018)
\end{botherref}
%
\OrigBibText
\begin{botherref}
\oauthor{\bsnm{Berschneider}, \binits{G.}},
\oauthor{\bsnm{B{\"o}ttcher}, \binits{B.}}:
On complex Gaussian random fields, Gaussian quadratic forms and sample distance
 multivariance.
arXiv: 1808.07280v1
(2018)
\end{botherref}
\endOrigBibText
\bptok{structpyb}%
\endbibitem

\bibitem{BickXu2009}
\begin{barticle}
\bauthor{\bsnm{Bickel}, \binits{P.J.}},
\bauthor{\bsnm{Xu}, \binits{Y.}}:
\batitle{{D}iscussion of: {B}rownian distance covariance}.
\bjtitle{The Annals of Applied Statistics}
\bvolume{3}(\bissue{4}),
\bfpage{1266}--\blpage{1269}
(\byear{2009})
\bid{doi={10.1214/09-AOAS312A}, mr={2752128}}
\end{barticle}
%
\OrigBibText
\begin{barticle}
\bauthor{\bsnm{Bickel}, \binits{P.J.}},
\bauthor{\bsnm{Xu}, \binits{Y.}}:
\batitle{{D}iscussion of: {B}rownian distance covariance}.
\bjtitle{The Annals of Applied Statistics}
\bvolume{3}(\bissue{4}),
\bfpage{1266}--\blpage{1269}
(\byear{2009})
\end{barticle}
\endOrigBibText
\bptok{structpyb}%
\endbibitem

\bibitem{Boet2017}
\begin{botherref}
\oauthor{\bsnm{B{\"o}ttcher}, \binits{B.}}:
Dependence structures - estimation and visualization using distance
 multivariance.
arXiv: \arxivhref{1712.06532v1}
(2017)
\end{botherref}
%
\OrigBibText
\begin{botherref}
\oauthor{\bsnm{B{\"o}ttcher}, \binits{B.}}:
Dependence structures - estimation and visualization using distance
 multivariance.
arXiv: 1712.06532v1
(2017)
\end{botherref}
\endOrigBibText
\bptok{structpyb}%
\endbibitem

\bibitem{Boett2017R-1.0.5}
\begin{botherref}
\oauthor{\bsnm{B{\"o}ttcher}, \binits{B.}}:
Multivariance: {Measuring Multivariate Dependence Using Distance
 Multivariance}.
(2017).
R package version 1.0.5
\end{botherref}
%
\OrigBibText
\begin{botherref}
\oauthor{\bsnm{B{\"o}ttcher}, \binits{B.}}:
Multivariance:{Measuring Multivariate Dependence Using Distance
 Multivariance}.
(2017).
R package version 1.0.5
\end{botherref}
\endOrigBibText
\bptok{structpyb}%
\endbibitem

\bibitem{part2}
\begin{botherref}
\oauthor{\bsnm{B{\"o}ttcher}, \binits{B.}},
\oauthor{\bsnm{Keller-Ressel}, \binits{M.}},
\oauthor{\bsnm{Schilling}, \binits{R.L.}}:
Distance Multivariance: New dependence measures for random vectors (submitted).
Revised version of arXiv: \arxivhref{1711.07775v1}
(2017)
\end{botherref}
%
\OrigBibText
\begin{botherref}
\oauthor{\bsnm{B{\"o}ttcher}, \binits{B.}},
\oauthor{\bsnm{Keller-Ressel}, \binits{M.}},
\oauthor{\bsnm{Schilling}, \binits{R.L.}}:
Distance Multivariance: New dependence measures for random vectors (submitted).
Revised version of arXiv: 1711.07775v1
(2017)
\end{botherref}
\endOrigBibText
\bptok{structpyb}%
\endbibitem

\bibitem{BoetSchiWang2013}
\begin{bbook}
\bauthor{\bsnm{B{\"o}ttcher}, \binits{B.}},
\bauthor{\bsnm{Schilling}, \binits{R.L.}},
\bauthor{\bsnm{Wang}, \binits{J.}}:
\bbtitle{{L}\'evy-{T}ype {P}rocesses: {C}onstruction, {A}pproximation and
 {S}ample {P}ath {P}roperties}.
\bsertitle{Lecture Notes in Mathematics, L\'evy Matters},
vol.~\bseriesno{2099}.
\bpublisher{Springer}
(\byear{2013})
\bid{doi={10.1007/978-3-319-02684-8}, mr={3156646}}
\end{bbook}
%
\OrigBibText
\begin{bbook}
\bauthor{\bsnm{B{\"o}ttcher}, \binits{B.}},
\bauthor{\bsnm{Schilling}, \binits{R.L.}},
\bauthor{\bsnm{Wang}, \binits{J.}}:
\bbtitle{{L}\'evy-{T}ype {P}rocesses: {C}onstruction, {A}pproximation and
 {S}ample {P}ath {P}roperties}.
\bsertitle{Lecture Notes in Mathematics, L\'evy Matters},
vol.~\bseriesno{2099}.
\bpublisher{Springer}
(\byear{2013})
\end{bbook}
\endOrigBibText
\bptok{structpyb}%
\endbibitem

\bibitem{Cope2009}
\begin{barticle}
\bauthor{\bsnm{Cope}, \binits{L.}}:
\batitle{{D}iscussion of: {B}rownian distance covariance}.
\bjtitle{The Annals of Applied Statistics}
\bvolume{3}(\bissue{4}),
\bfpage{1279}--\blpage{1281}
(\byear{2009}).
doi:\doiurl{10.1214/00-AOAS312C}
\bid{doi={10.1214/00-AOAS312C}, mr={2752130}}
\end{barticle}
%
\OrigBibText
\begin{barticle}
\bauthor{\bsnm{Cope}, \binits{L.}}:
\batitle{{D}iscussion of: {B}rownian distance covariance}.
\bjtitle{The Annals of Applied Statistics}
\bvolume{3}(\bissue{4}),
\bfpage{1279}--\blpage{1281}
(\byear{2009}).
doi:\doiurl{10.1214/00-AOAS312C}
\end{barticle}
\endOrigBibText
\bptok{structpyb}%
\endbibitem

\bibitem{Csor1981}
\begin{botherref}
\oauthor{\bsnm{Cs{\"o}rg{\H{o}}}, \binits{S.}}:
{L}imit behaviour of the empirical characteristic function.
The Annals of Probability,
130--144
(1981)
\end{botherref}
\bid{mr={0606802}}
%
\OrigBibText
\begin{botherref}
\oauthor{\bsnm{Cs{\"o}rg{\H{o}}}, \binits{S.}}:
{L}imit behaviour of the empirical characteristic function.
The Annals of Probability,
130--144
(1981)
\end{botherref}
\endOrigBibText
\bptok{structpyb}%
\endbibitem

\bibitem{Csoe1981a}
\begin{barticle}
\bauthor{\bsnm{Cs{\"o}rg{\H{o}}}, \binits{S.}}:
\batitle{Multivariate empirical characteristic functions}.
\bjtitle{Zeitschrift f{\"u}r Wahrscheinlichkeitstheorie und Verwandte Gebiete}
\bvolume{55}(\bissue{2}),
\bfpage{203}--\blpage{229}
(\byear{1981}).
\bid{doi={10.1007/BF00535160}, mr={0608017}}
\end{barticle}
%
\OrigBibText
\begin{barticle}
\bauthor{\bsnm{Cs{\"o}rg{\H{o}}}, \binits{S.}}:
\batitle{Multivariate empirical characteristic functions}.
\bjtitle{Zeitschrift f{\"u}r Wahrscheinlichkeitstheorie und Verwandte Gebiete}
\bvolume{55}(\bissue{2}),
\bfpage{203}--\blpage{229}
(\byear{1981}).
doi:\doiurl{10.1007/BF00535160}
\end{barticle}
\endOrigBibText
\bptok{structpyb}%
\endbibitem

\bibitem{Csoe1985}
\begin{barticle}
\bauthor{\bsnm{Cs{\"o}rg{\H{o}}}, \binits{S.}}:
\batitle{{T}esting for independence by the empirical characteristic function}.
\bjtitle{Journal of Multivariate Analysis}
\bvolume{16}(\bissue{3}),
\bfpage{290}--\blpage{299}
(\byear{1985})
\bid{doi={10.1016/0047-\\259X(85)90022-3}, mr={0793494}}
\end{barticle}
%
\OrigBibText
\begin{barticle}
\bauthor{\bsnm{Cs{\"o}rg{\H{o}}}, \binits{S.}}:
\batitle{{T}esting for independence by the empirical characteristic function}.
\bjtitle{Journal of Multivariate Analysis}
\bvolume{16}(\bissue{3}),
\bfpage{290}--\blpage{299}
(\byear{1985})
\end{barticle}
\endOrigBibText
\bptok{structpyb}%
\endbibitem

\bibitem{Feue2009}
\begin{barticle}
\bauthor{\bsnm{Feuerverger}, \binits{A.}}:
\batitle{{D}iscussion of: {B}rownian distance covariance}.
\bjtitle{The Annals of Applied Statistics}
\bvolume{3}(\bissue{4}),
\bfpage{1282}--\blpage{1284}
(\byear{2009})
\bid{doi={10.1214/09-AOAS312D}, mr={2752131}}
\end{barticle}
%
\OrigBibText
\begin{barticle}
\bauthor{\bsnm{Feuerverger}, \binits{A.}}:
\batitle{{D}iscussion of: {B}rownian distance covariance}.
\bjtitle{The Annals of Applied Statistics}
\bvolume{3}(\bissue{4}),
\bfpage{1282}--\blpage{1284}
(\byear{2009})
\end{barticle}
\endOrigBibText
\bptok{structpyb}%
\endbibitem

\bibitem{FeueMure1977}
\begin{botherref}
\oauthor{\bsnm{Feuerverger}, \binits{A.}},
\oauthor{\bsnm{Mureika}, \binits{R.A.}}:
{T}he empirical characteristic function and its applications.
The Annals of Statistics,
88--97
(1977)
\end{botherref}
\bid{mr={0428584}}
%
\OrigBibText
\begin{botherref}
\oauthor{\bsnm{Feuerverger}, \binits{A.}},
\oauthor{\bsnm{Mureika}, \binits{R.A.}}:
{T}he empirical characteristic function and its applications.
The Annals of Statistics,
88--97
(1977)
\end{botherref}
\endOrigBibText
\bptok{structpyb}%
\endbibitem

\bibitem{Geno2009}
\begin{barticle}
\bauthor{\bsnm{Genovese}, \binits{C.R.}}:
\batitle{{D}iscussion of: {B}rownian distance covariance}.
\bjtitle{The Annals of Applied Statistics}
\bvolume{3}(\bissue{4}),
\bfpage{1299}--\blpage{1302}
(\byear{2009}).
doi:\doiurl{10.1214/09-AOAS312G}
\bid{doi={10.1214/09-AOAS312G}, mr={2752134}}
\end{barticle}
%
\OrigBibText
\begin{barticle}
\bauthor{\bsnm{Genovese}, \binits{C.R.}}:
\batitle{{D}iscussion of: {B}rownian distance covariance}.
\bjtitle{The Annals of Applied Statistics}
\bvolume{3}(\bissue{4}),
\bfpage{1299}--\blpage{1302}
(\byear{2009}).
doi:\doiurl{10.1214/09-AOAS312G}
\end{barticle}
\endOrigBibText
\bptok{structpyb}%
\endbibitem

\bibitem{GretFukuSrip2009}
\begin{barticle}
\bauthor{\bsnm{Gretton}, \binits{A.}},
\bauthor{\bsnm{Fukumizu}, \binits{K.}},
\bauthor{\bsnm{Sriperumbudur}, \binits{B.K.}}:
\batitle{{D}iscussion of: {B}rownian distance covariance}.
\bjtitle{The Annals of Applied Statistics}
\bvolume{3}(\bissue{4}),
\bfpage{1285}--\blpage{1294}
(\byear{2009}).
\bid{doi={10.1214/09-AOAS312E}, mr={2752132}}
\end{barticle}
%
\OrigBibText
\begin{barticle}
\bauthor{\bsnm{Gretton}, \binits{A.}},
\bauthor{\bsnm{Fukumizu}, \binits{K.}},
\bauthor{\bsnm{Sriperumbudur}, \binits{B.K.}}:
\batitle{{D}iscussion of: {B}rownian distance covariance}.
\bjtitle{The Annals of Applied Statistics}
\bvolume{3}(\bissue{4}),
\bfpage{1285}--\blpage{1294}
(\byear{2009}).
doi:\doiurl{10.1214/09-AOAS312E}
\end{barticle}
\endOrigBibText
\bptok{structpyb}%
\endbibitem

\bibitem{Jaco2001}
\begin{bbook}
\bauthor{\bsnm{Jacob}, \binits{N.}}:
\bbtitle{{P}seudo-{D}ifferential {O}perators and {M}arkov {P}rocesses {I}.
 {F}ourier {A}nalysis and {S}emigroups}.
\bpublisher{Imperial College Press},
\blocation{London}
(\byear{2001})
\bid{doi={10.\\1142/9781860949746}, mr={1873235}}
\end{bbook}
%
\OrigBibText
\begin{bbook}
\bauthor{\bsnm{Jacob}, \binits{N.}}:
\bbtitle{{P}seudo-{D}ifferential {O}perators and {M}arkov {P}rocesses {I}.
 {F}ourier {A}nalysis and {S}emigroups}.
\bpublisher{Imperial College Press},
\blocation{London}
(\byear{2001})
\end{bbook}
\endOrigBibText
\bptok{structpyb}%
\endbibitem

\bibitem{JacoKnopLandSchi2011}
\begin{barticle}
\bauthor{\bsnm{Jacob}, \binits{N.}},
\bauthor{\bsnm{Knopova}, \binits{V.}},
\bauthor{\bsnm{Landwehr}, \binits{S.}},
\bauthor{\bsnm{Schilling}, \binits{R.L.}}:
\batitle{A geometric interpretation of the transition density of a symmetric
 {L}\'evy process}.
\bjtitle{Science China: Mathematics}
\bvolume{55},
\bfpage{1099}--\blpage{1126}
(\byear{2012})
\bid{doi={10.1007/s11425-012-4368-0}, mr={2925579}}
\end{barticle}
%
\OrigBibText
\begin{barticle}
\bauthor{\bsnm{Jacob}, \binits{N.}},
\bauthor{\bsnm{Knopova}, \binits{V.}},
\bauthor{\bsnm{Landwehr}, \binits{S.}},
\bauthor{\bsnm{Schilling}, \binits{R.L.}}:
\batitle{A geometric interpretation of the transition density of a symetric
 {L}\'evy process}.
\bjtitle{Science China: Mathematics}
\bvolume{55},
\bfpage{1099}--\blpage{1126}
(\byear{2012})
\end{barticle}
\endOrigBibText
\bptok{structpyb}%
\endbibitem

\bibitem{JinMatt2017}
\begin{botherref}
\oauthor{\bsnm{Jin}, \binits{Z.}},
\oauthor{\bsnm{Matteson}, \binits{D.S.}}:
Generalizing Distance Covariance to Measure and Test Multivariate Mutual
 Dependence.
arXiv: \arxivhref{1709.02532v1}
(2017)
\end{botherref}
%
\OrigBibText
\begin{botherref}
\oauthor{\bsnm{Jin}, \binits{Z.}},
\oauthor{\bsnm{Matteson}, \binits{D.S.}}:
Generalizing Distance Covariance to Measure and Test Multivariate Mutual
 Dependence.
arXiv: 1709.02532v1
(2017)
\end{botherref}
\endOrigBibText
\bptok{structpyb}%
\endbibitem

\bibitem{Koso2009}
\begin{barticle}
\bauthor{\bsnm{Kosorok}, \binits{M.R.}}:
\batitle{{D}iscussion of: {B}rownian distance covariance}.
\bjtitle{The Annals of Applied Statistics}
\bvolume{3}(\bissue{4}),
\bfpage{1270}--\blpage{1278}
(\byear{2009}).
doi:\doiurl{10.1214/09-AOAS312B}
\bid{doi={10.1214/09-AOAS312B}, mr={2752129}}
\end{barticle}
%
\OrigBibText
\begin{barticle}
\bauthor{\bsnm{Kosorok}, \binits{M.R.}}:
\batitle{{D}iscussion of: {B}rownian distance covariance}.
\bjtitle{The Annals of Applied Statistics}
\bvolume{3}(\bissue{4}),
\bfpage{1270}--\blpage{1278}
(\byear{2009}).
doi:\doiurl{10.1214/09-AOAS312B}
\end{barticle}
\endOrigBibText
\bptok{structpyb}%
\endbibitem

\bibitem{Lyon2013}
\begin{barticle}
\bauthor{\bsnm{Lyons}, \binits{R.}}:
\batitle{{D}istance covariance in metric spaces}.
\bjtitle{The Annals of Probability}
\bvolume{41}(\bissue{5}),
\bfpage{3284}--\blpage{3305}
(\byear{2013})
\bid{doi={10.1214/12-AOP803}, mr={3127883}}
\end{barticle}
%
\OrigBibText
\begin{barticle}
\bauthor{\bsnm{Lyons}, \binits{R.}}:
\batitle{{D}istance covariance in metric spaces}.
\bjtitle{The Annals of Probability}
\bvolume{41}(\bissue{5}),
\bfpage{3284}--\blpage{3305}
(\byear{2013})
\end{barticle}
\endOrigBibText
\bptok{structpyb}%
\endbibitem

\bibitem{Mura2001}
\begin{bchapter}
\bauthor{\bsnm{Murata}, \binits{N.}}:
\bctitle{{P}roperties of the empirical characteristic function and its
 application to testing for independence}.
In: \bbtitle{Third International Workshop on Independent Component Analysis and
 Signal Separation (ICA2001)},
pp.~\bfpage{295}--\blpage{300}
(\byear{2001})
\end{bchapter}
%
\OrigBibText
\begin{bchapter}
\bauthor{\bsnm{Murata}, \binits{N.}}:
\bctitle{{P}roperties of the empirical characteristic function and its
 application to testing for independence}.
In: \bbtitle{Third International Workshop on Independent Component Analysis and
 Signal Separation (ICA2001)},
pp.~\bfpage{295}--\blpage{300}
(\byear{2001})
\end{bchapter}
\endOrigBibText
\bptok{structpyb}%
\endbibitem

\bibitem{Newt2009}
\begin{barticle}
\bauthor{\bsnm{Newton}, \binits{M.A.}}:
\batitle{{I}ntroducing the discussion paper by {S}z{\'e}kely and {R}izzo}.
\bjtitle{The Annals of Applied Statistics}
\bvolume{3}(\bissue{4}),
\bfpage{1233}--\blpage{1235}
(\byear{2009}).
\bid{doi={10.1214/09-\\AOAS34INTRO}, mr={2752126}}
\end{barticle}
%
\OrigBibText
\begin{barticle}
\bauthor{\bsnm{Newton}, \binits{M.A.}}:
\batitle{{I}ntroducing the discussion paper by {S}z{\'e}kely and {R}izzo}.
\bjtitle{The Annals of Applied Statistics}
\bvolume{3}(\bissue{4}),
\bfpage{1233}--\blpage{1235}
(\byear{2009}).
doi:\doiurl{10.1214/09-AOAS34INTRO}
\end{barticle}
\endOrigBibText
\bptok{structpyb}%
\endbibitem

\bibitem{Remi2009}
\begin{barticle}
\bauthor{\bsnm{R{\'e}millard}, \binits{B.}}:
\batitle{{D}iscussion of: {B}rownian distance covariance}.
\bjtitle{The Annals of Applied Statistics}
\bvolume{3}(\bissue{4}),
\bfpage{1295}--\blpage{1298}
(\byear{2009}).
doi:\doiurl{10.1214/09-AOAS312F}
\bid{doi={10.1214/09-AOAS312F}, mr={2752133}}
\end{barticle}
%
\OrigBibText
\begin{barticle}
\bauthor{\bsnm{R{\'e}millard}, \binits{B.}}:
\batitle{{D}iscussion of: {B}rownian distance covariance}.
\bjtitle{The Annals of Applied Statistics}
\bvolume{3}(\bissue{4}),
\bfpage{1295}--\blpage{1298}
(\byear{2009}).
doi:\doiurl{10.1214/09-AOAS312F}
\end{barticle}
\endOrigBibText
\bptok{structpyb}%
\endbibitem

\bibitem{Sasv1994}
\begin{bbook}
\bauthor{\bsnm{Sasv{\'a}ri}, \binits{Z.}}:
\bbtitle{Positive {D}efinite and {D}efinitizable {F}unctions}.
\bpublisher{Akademie-Verlag},
\blocation{Berlin}
(\byear{1994})
\bid{mr={1270018}}
\end{bbook}
%
\OrigBibText
\begin{bbook}
\bauthor{\bsnm{Sasv{\'a}ri}, \binits{Z.}}:
\bbtitle{Positive {D}efinite and {D}efinitizable {F}unctions}.
\bpublisher{Akademie-Verlag},
\blocation{Berlin}
(\byear{1994})
\end{bbook}
\endOrigBibText
\bptok{structpyb}%
\endbibitem

\bibitem{sato}
\begin{bbook}
\bauthor{\bsnm{Sato}, \binits{K.}}:
\bbtitle{{L}{\'e}vy {P}rocesses and {I}nfinitely {D}ivisible {D}istributions}.
\bpublisher{Cambridge University Press},
\blocation{Cambridge}
(\byear{1999})
\bid{mr={1739520}}
\end{bbook}
%
\OrigBibText
\begin{bbook}
\bauthor{\bsnm{Sato}, \binits{K.}}:
\bbtitle{{L}{\'e}vy {P}rocesses and {I}nfinitely {D}ivisible {D}istributions}.
\bpublisher{Cambridge University Press},
\blocation{Cambridge}
(\byear{1999})
\end{bbook}
\endOrigBibText
\bptok{structpyb}%
\endbibitem

\bibitem{SchiSchn09}
\begin{barticle}
\bauthor{\bsnm{Schilling}, \binits{R.}},
\bauthor{\bsnm{Schnurr}, \binits{A.}}:
\batitle{{The symbol associated with the solution of a stochastic differential
 equation}}.
\bjtitle{Electronic Journal of Probability}
\bvolume{15},
\bfpage{1369}--\blpage{1393}
(\byear{2010})
\bid{doi={10.1214/EJP.v15-807}, mr={2721050}}
\end{barticle}
%
\OrigBibText
\begin{barticle}
\bauthor{\bsnm{Schilling}, \binits{R.}},
\bauthor{\bsnm{Schnurr}, \binits{A.}}:
\batitle{{The symbol associated with the solution of a stochastic differential
 equation}}.
\bjtitle{Electronic Journal of Probability}
\bvolume{15},
\bfpage{1369}--\blpage{1393}
(\byear{2010})
\end{barticle}
\endOrigBibText
\bptok{structpyb}%
\endbibitem

\bibitem{ssv}
\begin{bbook}
\bauthor{\bsnm{Schilling}, \binits{R.L.}},
\bauthor{\bsnm{Song}, \binits{R.}},
\bauthor{\bsnm{Vondra\v{c}ek}, \binits{Z.}}:
\bbtitle{{B}ernstein Functions. {T}heory and Applications},
\bedition{2nd} edn.
\bpublisher{de {G}ruyter}
(\byear{2012})
\bid{doi={10.1515/9783110269338}, mr={2978140}}
\end{bbook}
%
\OrigBibText
\begin{bbook}
\bauthor{\bsnm{Schilling}, \binits{R.L.}},
\bauthor{\bsnm{Song}, \binits{R.}},
\bauthor{\bsnm{Vondra\v{c}ek}, \binits{Z.}}:
\bbtitle{{B}ernstein Functions. {T}heory and Applications},
\bedition{2nd} edn.
\bpublisher{de {G}ruyter}
(\byear{2012})
\end{bbook}
\endOrigBibText
\bptok{structpyb}%
\endbibitem

\bibitem{Scho1938a}
\begin{barticle}
\bauthor{\bsnm{Schoenberg}, \binits{I.J.}}:
\batitle{{M}etric spaces and positive definite functions}.
\bjtitle{Transactions of the American Mathematical Society}
\bvolume{44}(\bissue{3}),
\bfpage{522}--\blpage{536}
(\byear{1938})
\bid{doi={10.\\2307/1989894}, mr={1501980}}
\end{barticle}
%
\OrigBibText
\begin{barticle}
\bauthor{\bsnm{Schoenberg}, \binits{I.J.}}:
\batitle{{M}etric spaces and positive definite functions}.
\bjtitle{Transactions of the American Mathematical Society}
\bvolume{44}(\bissue{3}),
\bfpage{522}--\blpage{536}
(\byear{1938})
\end{barticle}
\endOrigBibText
\bptok{structpyb}%
\endbibitem

\bibitem{Serf2009}
\begin{bbook}
\bauthor{\bsnm{Serfling}, \binits{R.J.}}:
\bbtitle{{A}pproximation {T}heorems of {M}athematical {S}tatistics}.
\bpublisher{John Wiley \& Sons}
(\byear{2009})
\bid{mr={0595165}}
\end{bbook}
%
\OrigBibText
\begin{bbook}
\bauthor{\bsnm{Serfling}, \binits{R.J.}}:
\bbtitle{{A}pproximation {T}heorems of {M}athematical {S}tatistics}.
\bpublisher{John Wiley \& Sons}
(\byear{2009})
\end{bbook}
\endOrigBibText
\bptok{structpyb}%
\endbibitem

\bibitem{SzekRizz2005}
\begin{barticle}
\bauthor{\bsnm{Sz{\'e}kely}, \binits{G.J.}},
\bauthor{\bsnm{Rizzo}, \binits{M.L.}}:
\batitle{Hierarchical clustering via joint between-within distances: Extending
 ward's minimum variance method}.
\bjtitle{Journal of {C}lassification}
\bvolume{22}(\bissue{2}),
\bfpage{151}--\blpage{183}
(\byear{2005})
\bid{doi={10.1007/s00357-005-0012-9}, mr={2231170}}
\end{barticle}
%
\OrigBibText
\begin{barticle}
\bauthor{\bsnm{Sz{\'e}kely}, \binits{G.J.}},
\bauthor{\bsnm{Rizzo}, \binits{M.L.}}:
\batitle{Hierarchical clustering via joint between-within distances: Extending
 ward's minimum variance method}.
\bjtitle{Journal of {C}lassification}
\bvolume{22}(\bissue{2}),
\bfpage{151}--\blpage{183}
(\byear{2005})
\end{barticle}
\endOrigBibText
\bptok{structpyb}%
\endbibitem

\bibitem{SzekRizz2009}
\begin{barticle}
\bauthor{\bsnm{Sz{\'e}kely}, \binits{G.J.}},
\bauthor{\bsnm{Rizzo}, \binits{M.L.}}:
\batitle{{B}rownian distance covariance}.
\bjtitle{The Annals of Applied Statistics}
\bvolume{3}(\bissue{4}),
\bfpage{1236}--\blpage{1265}
(\byear{2009})
\bid{doi={10.1214/09-AOAS312}, mr={2752127}}
\end{barticle}
%
\OrigBibText
\begin{barticle}
\bauthor{\bsnm{Sz{\'e}kely}, \binits{G.J.}},
\bauthor{\bsnm{Rizzo}, \binits{M.L.}}:
\batitle{{B}rownian distance covariance}.
\bjtitle{The Annals of Applied Statistics}
\bvolume{3}(\bissue{4}),
\bfpage{1236}--\blpage{1265}
(\byear{2009})
\end{barticle}
\endOrigBibText
\bptok{structpyb}%
\endbibitem

\bibitem{SzekRizz2009a}
\begin{barticle}
\bauthor{\bsnm{Sz{\'e}kely}, \binits{G.J.}},
\bauthor{\bsnm{Rizzo}, \binits{M.L.}}:
\batitle{{R}ejoinder: {B}rownian distance covariance}.
\bjtitle{The Annals of Applied Statistics}
\bvolume{3}(\bissue{4}),
\bfpage{1303}--\blpage{1308}
(\byear{2009}).
doi:\doiurl{10.1214/09-AOAS312REJ}
\bid{doi={10.1214/09-AOAS312REJ}, mr={2752135}}
\end{barticle}
%
\OrigBibText
\begin{barticle}
\bauthor{\bsnm{Sz{\'e}kely}, \binits{G.J.}},
\bauthor{\bsnm{Rizzo}, \binits{M.L.}}:
\batitle{{R}ejoinder: {B}rownian distance covariance}.
\bjtitle{The Annals of Applied Statistics}
\bvolume{3}(\bissue{4}),
\bfpage{1303}--\blpage{1308}
(\byear{2009}).
doi:\doiurl{10.1214/09-AOAS312REJ}
\end{barticle}
\endOrigBibText
\bptok{structpyb}%
\endbibitem

\bibitem{SzekRizz2012}
\begin{barticle}
\bauthor{\bsnm{Sz{\'e}kely}, \binits{G.J.}},
\bauthor{\bsnm{Rizzo}, \binits{M.L.}}:
\batitle{{O}n the uniqueness of distance covariance}.
\bjtitle{Statistics and Probability Letters}
\bvolume{82}(\bissue{12}),
\bfpage{2278}--\blpage{2282}
(\byear{2012})
\bid{doi={10.1016/j.spl.\\2012.08.007}, mr={2979766}}
\end{barticle}
%
\OrigBibText
\begin{barticle}
\bauthor{\bsnm{Sz{\'e}kely}, \binits{G.J.}},
\bauthor{\bsnm{Rizzo}, \binits{M.L.}}:
\batitle{{O}n the uniqueness of distance covariance}.
\bjtitle{Statistics and Probability Letters}
\bvolume{82}(\bissue{12}),
\bfpage{2278}--\blpage{2282}
(\byear{2012})
\end{barticle}
\endOrigBibText
\bptok{structpyb}%
\endbibitem

\bibitem{SzekRizzBaki2007}
\begin{barticle}
\bauthor{\bsnm{Sz{\'e}kely}, \binits{G.J.}},
\bauthor{\bsnm{Rizzo}, \binits{M.L.}},
\bauthor{\bsnm{Bakirov}, \binits{N.K.}}:
\batitle{{M}easuring and testing dependence by correlation of distances}.
\bjtitle{The Annals of Statistics}
\bvolume{35}(\bissue{6}),
\bfpage{2769}--\blpage{2794}
(\byear{2007})
\bid{doi={10.1214/009053607000000505}, mr={2382665}}
\end{barticle}
%
\OrigBibText
\begin{barticle}
\bauthor{\bsnm{Sz{\'e}kely}, \binits{G.J.}},
\bauthor{\bsnm{Rizzo}, \binits{M.L.}},
\bauthor{\bsnm{Bakirov}, \binits{N.K.}}:
\batitle{{M}easuring and testing dependence by correlation of distances}.
\bjtitle{The Annals of Statistics}
\bvolume{35}(\bissue{6}),
\bfpage{2769}--\blpage{2794}
(\byear{2007})
\end{barticle}
\endOrigBibText
\bptok{structpyb}%
\endbibitem

\bibitem{Usha1999}
\begin{bbook}
\bauthor{\bsnm{Ushakov}, \binits{N.G.}}:
\bbtitle{{S}elected {T}opics in {C}haracteristic {F}unctions}.
\bpublisher{VSP}
(\byear{1999})
\bid{doi={10.1515/9783110935981}, mr={1745554}}
\end{bbook}
%
\OrigBibText
\begin{bbook}
\bauthor{\bsnm{Ushakov}, \binits{N.G.}}:
\bbtitle{{S}elected {T}opics in {C}haracteristic {F}unctions}.
\bpublisher{VSP}
(\byear{1999})
\end{bbook}
\endOrigBibText
\bptok{structpyb}%
\endbibitem

\bibitem{WittMuel1995}
\begin{bbook}
\bauthor{\bsnm{Witting}, \binits{H.}},
\bauthor{\bsnm{M{\"u}ller-Funk}, \binits{U.}}:
\bbtitle{{M}athematische {S}tatistik {II}}.
\bpublisher{Teubner, Stuttgart}
(\byear{1995})
\bid{doi={10.1007/978-3-322-90152-1}, mr={1363716}}
\end{bbook}
%
\OrigBibText
\begin{bbook}
\bauthor{\bsnm{Witting}, \binits{H.}},
\bauthor{\bsnm{M{\"u}ller-Funk}, \binits{U.}}:
\bbtitle{{M}athematische {S}tatistik {II}}.
\bpublisher{Teubner, Stuttgart}
(\byear{1995})
\end{bbook}
\endOrigBibText
\bptok{structpyb}%
\endbibitem

\bibitem{zast00}
\begin{barticle}
\bauthor{\bsnm{Zastavnyi}, \binits{V.P.}}:
\batitle{{O}n positive definiteness of some functions}.
\bjtitle{Journal of Multivariate Analysis}
\bvolume{73}(\bissue{3}),
\bfpage{55}--\blpage{81}
(\byear{2000})
\bid{doi={10.1006/jmva.1999.1864}, mr={1766121}}
\end{barticle}
%
\OrigBibText
\begin{barticle}
\bauthor{\bsnm{Zastavnyi}, \binits{V.P.}}:
\batitle{{O}n positive definiteness of some functions}.
\bjtitle{Journal of Multivariate Analysis}
\bvolume{73}(\bissue{3}),
\bfpage{55}--\blpage{81}
(\byear{2000})
\end{barticle}
\endOrigBibText
\bptok{structpyb}%
\endbibitem

\end{thebibliography}




\end{document}